\documentclass[reqno]{amsart}

\usepackage{verbatim} 
\usepackage{amsmath}
\usepackage{amsfonts}
\usepackage{amssymb}
\usepackage{mathrsfs}
\usepackage{amsthm}
\usepackage{newlfont}

\usepackage{fullpage,amsmath, amssymb,amscd}
\usepackage{latexsym, epsfig, color}
\usepackage[mathscr]{eucal}
\usepackage{xypic, graphicx}
\usepackage{amscd}
\usepackage[all, cmtip]{xy}

\usepackage[textwidth=50,textsize=tiny]{todonotes}
\usepackage{colonequals}

\newcommand\cyr{%
 \renewcommand\rmdefault{wncyr}%
 \renewcommand\sfdefault{wncyss}%
 \renewcommand\encodingdefault{OT2}%
\normalfont\selectfont} \DeclareTextFontCommand{\textcyr}{\cyr}

\newtheorem{theorem}{Theorem}
\newtheorem{lemma}[theorem]{Lemma}
\newtheorem{corollary}[theorem]{Corollary}
\newtheorem{proposition}[theorem]{Proposition}
\newtheorem{remark}[theorem]{Remark}

\newtheorem{conjecture}[theorem]{Conjecture}
\newtheorem{definition}[theorem]{Definition}

\def\Z{\mathbb Z}
\def\N{\mathbb N}
\def\Q{\mathbb Q}
\def\R{\mathbb R}
\def\C{\mathbb C}
\def\F{\mathbb F}

\def\mbz{\mathbb Z}
\def\mbq{\mathbb Q}

\def\Ker{\operatorname{Ker}}
\def\Im{\operatorname{Im}}
\def\Re{\operatorname{Re}}

\def\End{\operatorname{End}}

\def\Aut{\operatorname{Aut}}
\def\Gal{\operatorname{Gal}}

\def\mod{\operatorname{mod}}

\def\disc{\operatorname{disc}}

\def\exp{\operatorname{exp}}
\def\lcm{\operatorname{lcm}}
\def\gcd{\operatorname{gcd}}

\def\GL{\operatorname{GL}}

\def\li{\operatorname{li}}

\def\O{\operatorname{O}}
\def\o{\operatorname{o}}

\def\log{\operatorname{log}}

\def\rad{\operatorname{rad}}

\def\ds{\displaystyle}

\def\pr{\operatorname{pr}}
\def\Res{\operatorname{Res}}

\newcommand{\gd}{\delta}
\newcommand{\gD}{\Delta}
\newcommand{\gs}{\sigma}
\newcommand{\mbc}{\mathbb{C}}

\newcommand{\ol}[1]{\overline{#1}}

\begin{document}

\title{
Constants in Titchmarsh divisor problems for elliptic curves
}

\date{June 11, 2017}


\author{
Renee Bell
}
\address[Renee Bell]{
\begin{itemize}
\item[-]
Department of Mathematics, Massachusetts Institute of Technology, 77 Massachusetts Ave, Cambridge, MA 02139
\end{itemize}
} \email[Renee Bell]{rhbell@math.mit.edu}

\author{
Clifford Blakestad
}
\address[Clifford Blackestad]{
\begin{itemize}
\item[-]
Department of Mathematics, University of Colorado Boulder, 395 UCB, Boulder, CO 80309
\end{itemize}
} \email[Clifford Blackestad]{Clifford.Blakestad@colorado.edu}

\author{
Alina Carmen Cojocaru
}
\address[Alina Carmen  Cojocaru]{
\begin{itemize}
\item[-]
Department of Mathematics, Statistics and Computer Science, University of Illinois at Chicago, 851 S Morgan St, 322
SEO, Chicago, 60607, IL, USA;
\item[-]
Institute of Mathematics  ``Simion Stoilow'' of the Romanian Academy, 21 Calea Grivitei St, Bucharest, 010702,
Sector 1, Romania
\end{itemize}
} \email[Alina Carmen  Cojocaru]{cojocaru@uic.edu}

\author{
Alexander Cowan
}
\address[Alexander Cowan]{
\begin{itemize}
\item[-]
Department of Mathematics, Columbia University, 2990 Broadway, New York, NY 10027
\end{itemize}
} \email[Alexander Cowan]{cowan@math.columbia.edu}

\author{
Nathan Jones
}
\address[Nathan Jones]{
\begin{itemize}
\item[-]
Department of Mathematics, Statistics and Computer Science, University of Illinois at Chicago, 851 S Morgan St, 322
SEO, Chicago, 60607, IL, USA
\end{itemize}
} \email[Nathan Jones]{ncjones@uic.edu}

\author{
Vlad Matei
}
\address[Vlad Matei]{
\begin{itemize}
\item[-]
Department of Mathematics, University of Wisconsin Madison, 480 Lincoln Drive, Madison, WI 53706
\end{itemize}
} \email[Vlad Matei]{matei@wisc.edu}

\author{
Geoffrey Smith
}
\address[Geoffrey Smith]{
\begin{itemize}
\item[-]
Department of Mathematics, Harvard University, 1 Oxford St, Cambridge, MA 02138
\end{itemize}
} \email[Geoffrey Smith]{gsmith@math.harvard.edu}

\author{
Isabel Vogt
}
\address[Isabel Vogt]{
\begin{itemize}
\item[-]
Department of Mathematics, Massachusetts Institute of Technology, 77 Massachusetts Ave, Cambridge, MA 02139
\end{itemize}
} \email[Isabel Vogt]{ivogt@mit.edu}

\thanks{
ACC's  work on this material was partially supported  by the Simons Collaboration Grant under Award No. 318454. IV's work was partially supported by the NSF Graduate Research Fellowship Program and grant DMS-1601946.
}

\begin{abstract}
Inspired by the analogy between the group of units $\F_p^{\times}$ of the finite field with $p$ elements and the group of points
$E(\F_p)$ of an elliptic curve $E/\F_p$, E. Kowalski, A. Akbary \& D. Ghioca, and  T. Freiberg \& P. Kurlberg investigated the asymptotic behaviour 
of elliptic curve sums analogous to the Titchmarsh divisor sum $\sum_{p \leq x} \tau(p + a) \sim C x$.
In this paper, we present a comprehensive study of the constants $C(E)$  emerging in the asymptotic study of these elliptic curve divisor sums.
Specifically, by analyzing the division fields of an elliptic curve $E/\Q$, we prove upper bounds for the constants $C(E)$  and, in the generic case of a Serre curve, we prove explicit closed formulae for $C(E)$ amenable to concrete computations.  Moreover, we compute the moments of the constants  $C(E)$ 
over two-parameter families of elliptic curves $E/\Q$. 
Our methods and results complement recent studies of average constants occurring in other conjectures about reductions of elliptic curves 
by addressing not only the average behaviour, but also the individual behaviour of these constants, and by providing explicit tools towards the computational verifications of the expected asymptotics.
\end{abstract}

\maketitle


\bigskip
\section{Introduction}

The Titchmarsh divisor problem concerns the asymptotic behaviour of the sum
$\ds\sum_{p \leq x} \tau(p + a)$, as a function of $x$,
where $p$ denotes a rational prime, $\tau(n) := \#\{d \geq 1: d \mid n\}$ denotes the divisor function, and $a$ denotes a fixed integer. 
The study of this sum  
has spanned over five decades
and
is intimately related to some of the most significant research in analytic number theory 
(see \cite{BoFrIw}, \cite{Fo}, \cite{Ha}, \cite{Li},  and \cite{Ti}).
By results about primes in arithmetic progressions which have become standard, we now know that, as $x \rightarrow \infty$,
\begin{equation}\label{titchmarsh}
\ds\sum_{p \leq x} \tau(p + a)
\sim
\frac{\zeta(2) \zeta(3)}{\zeta(6)} 
\ds\prod_{\ell \mid a} 
\left(1
-
\frac{\ell}{\ell^2 - \ell + 1}\right)
x,
\end{equation}
where $\zeta(\cdot)$ denotes the Riemann zeta function and $\ell$ denotes a rational prime.

Divisor problems similar to Titchmarsh's may  be formulated  in other settings, such as the setting of elliptic curves, as we now describe.
Let $E/\Q$ be an elliptic curve defined over the field of rational numbers.
 For a prime $p$ of good reduction for $E$, let $\overline{E}/\F_p$ be the reduction of $E$ modulo $p$.  
 From the basic theory of elliptic curves (see \cite[Ch. III, \S 6]{Si}), it is known that
 the group $\overline{E}(\F_p)$ may be expressed as a product of two cyclic groups,
 \begin{equation*}
 \overline{E}(\F_p) \simeq \Z/d_{1, p} \Z \times \Z/d_{2, p} \Z,
 \end{equation*}
 where 
 $d_{1, p} = d_{1, p} (E), d_{2, p} = d_{2, p} (E)$
 are uniquely determined positive integers 
 satisfying
 $ d_{1, p} \mid d_{2, p}$.
Determining the asymptotic behaviour of
sums over $p \leq x$ of  arithmetic functions evaluated at the elementary divisors  $d_{1, p}$ and $d_{2, p}$
may be viewed as 
Titchmarsh divisor problems for elliptic curves. 
Such problems
unravel
striking similarities, 
but also
intriguing contrasts,
to the original Titchmarsh divisor problem,
as illustrated in
\cite{AkGh},
\cite{AkFe},
\cite{Co},
\cite{CoMu},
\cite{Fe},
\cite{FeMu},
\cite{FrKu},
\cite{FrPo},
\cite{Ki},
\cite{Ko},
\cite{Po},
and
\cite{Wu}.

The focus of our paper is on the constants emerging in the following three  Titchmarsh divisor problems for elliptic curves. In all the expressions below, the letters $p$ and $\ell$ denote primes, with $p$ being a prime of good reduction for the given elliptic curve.

\begin{conjecture}\label{d-one}  (Kowalski \cite[Section 3.2]{Ko})

\noindent
Let $E/\Q$ be an elliptic curve.
Then, as $x \rightarrow \infty$,
\begin{align}
\ds\sum_{p \leq x} d_{1, p} & \sim C_{d_1, \text{non-CM}}(E) \  \li (x), & \text{if } E \text{ is without complex multiplication,} \\
\ds\sum_{p \leq x} d_{1, p} & \sim C_{d_1, \text{CM}}(E) \  x, & \text{if }E \text{ is with complex multiplication,} 
\end{align}
where
\begin{align}
C_{d_1, \text{non-CM}}(E) &\colonequals \ds\sum_{m \geq 1} \frac{\phi(m)}{[\Q(E[m]): \Q]}, \\
C_{d_1, \text{CM}}(E) &\colonequals  \Res_{s=0} \sum_{m \geq 1} \frac{\phi(m)}{[\Q(E[m]): \Q]} \ m^{-s} ,
\end{align}
with 
$\phi(m) := \#\{1 \leq k \leq m: (k, m) =1\}$ denoting the Euler function of $m$,
$\Q(E[m])$ denoting the $m$-division field of $E$,
and 
$\li(x)\colonequals \ds\int_2^x \frac{1}{\log{t}} \ dt$ denoting the standard logarithmic integral.
\end{conjecture}

\begin{conjecture}\label{tau-d-one}(Akbary-Ghioca \cite[Section 1]{AkGh})

\noindent
Let $E/\Q$ be an elliptic curve, 
with or without complex multiplication.
Then, as $x \rightarrow \infty$,
\begin{equation}\label{tau-d-one-asymp}
\ds\sum_{p \leq x} \tau(d_{1, p}) \sim C_{\tau(d_1)}(E) \  \li (x),
\end{equation}
where
\begin{equation}
C_{\tau(d_1)}(E) := \ds\sum_{m \geq 1} \frac{1}{[\Q(E[m]): \Q]}.
\end{equation}
\end{conjecture}

\begin{conjecture}\label{d-two} (Freiberg-Kurlberg \cite[Section 1]{FrKu})

\noindent
Let $E/\Q$ be an elliptic curve, 
with or without complex multiplication.
Then, as $x \rightarrow \infty$,
\begin{equation}\label{d-two-asymp}
\ds\sum_{p \leq x } d_{2, p} \sim \frac{1}{2} C_{d_2}(E) \  \li \left(x^2\right),
\end{equation}
where
\begin{equation}
C_{d_2}(E) := \ds\sum_{m \geq 1} 
\frac{(-1)^{\omega(m)} \phi(\rad{m}) }{m[\Q(E[m]): \Q]},
\end{equation}
with 
$\omega(m) := \ds\sum_{\ell \mid m} 1$ denoting the number of distinct prime factors of $m$ 
and $\rad(m) := \ds\prod_{\ell \mid m} \ell$ denoting the product of distinct prime factors of $m$.
\end{conjecture}

The constants $C_{d_1, \text{non-CM}}(E)$, $C_{d_1, \text{CM}}(E)$, $C_{\tau(d_1)}(E)$, and $C_{d_2}(E)$
appearing in these conjectures are deeply related to the arithmetic of the elliptic curve $E/\Q$ and are heuristically derived via the Chebotarev density theorem by considering the action of a Frobenius element at $p$ on $E[m]$.

Conjecture \ref{d-one} was investigated by  Freiberg and  Pollack in \cite{FrPo} in the case that $E$ has complex multiplication; precisely,
 they proved that $\ds\sum_{p \leq x} d_{1,p} \asymp_E x$.   A similar result is not yet known if $E$ does not have complex multiplication, even under the Generalized Riemann Hypothesis.
Conjecture \ref{tau-d-one} was investigated by Akbary and  Ghioca in \cite{AkGh};
 precisely,  they proved  (\ref{tau-d-one-asymp}) under the Generalized Riemann Hypothesis if $E$ is without complex multiplication  and unconditionally if $E$ is with complex multiplication.
Conjecture \ref{d-two} was  investigated by Freiberg and  Kurlberg in \cite{FrKu};
precisely, they proved  (\ref{d-two-asymp}) under the Generalized Riemann Hypothesis if $E$ is without complex multiplication
 and unconditionally if $E$ is with complex multiplication.
The proofs of the main results in \cite{AkGh} and \cite{FrKu} rely upon the methods of \cite{CoMu} and have been refined in several subsequent works, including 
\cite{AkFe},  \cite{Fe}, \cite{FeMu}, \cite{Ki},  and \cite{Wu}.
 
Using ideas originating in \cite{FoMu}, Conjectures \ref{d-one} - \ref{d-two} may also  be investigated on average over elliptic curves in families.
For this,
we consider
parameters $A, B > 2$,
the family 
${\mathcal{C}}(A, B)$
of $\Q$-isomorphism classes of elliptic curves 
$E = E_{a, b}$ defined by 
$Y^2 = X^3 + a X + b$
with
$a, b \in \Z$,
$|a| \leq A, |b| \leq B$,  
$\Delta(E) = \Delta(a, b) := -16 (4a^3 + 27b^2) \neq 0$,
and  average the Titchmarsh divisor sums of Conjectures \ref{d-one} - \ref{d-two} over $E \in {\mathcal{C}}(A, B)$.

In this context, it is natural to consider the following idealized versions of the constants $C_{d_1, \text{non-CM}}(E)$, $C_{\tau(d_1)}(E)$, and $C_{d_2}(E)$,
\begin{eqnarray}\label{c-d-one}
C_{d_1} &:=& \sum_{m \geq 1} \frac{\phi(m)}{| \GL_2(\mbz/m\mbz) |}  = 
\ds\prod_{\ell}
\left(
1 + \frac{\ell^2}{(\ell^2 - 1) (\ell^3 - 1)}
\right)
=
1.25844835...,
\\
\label{c-tau-d-one}
C_{\tau(d_1)} &:=&  \sum_{m \geq 1} \frac{1}{| \GL_2(\mbz/m\mbz) |}  = 
\ds\prod_{\ell}
\left(
1 + \frac{\ell^3}{(\ell -1) (\ell^2 - 1) (\ell^4 - 1)}
\right)
=
1.2059016...,
\\
\label{c-d-two}
C_{d_2} &:=&   \sum_{m \geq 1} \frac{(-1)^{(\omega(m))} \phi(\rad(m))}{m | \GL_2(\mbz/m\mbz) |}  = 
\ds\prod_{\ell}
\left(
1 + \frac{\ell^3}{(\ell^2 - 1) (\ell^5 - 1)}
\right)
=
0.89922825....
\end{eqnarray}
These {\it{idealized constants}} turn out to be the {\it{average constants}}, as explained below.

In \cite[Cor 1.6]{AkFe},
Akbary and Felix proved that
for any $c  > 1$ and $x > e$, there exists $c_1 >0 $ such that,
 for any $A = A(x)$, $B = B(x)$ with 
 $A, B > \exp(c_1(\log{x})^{1/2})$
 and
 $A  B >  x(\log{x})^{4 + 2c}$,
 we have
\begin{equation} \label{tauofd1onaverage}
\frac{1}{|\mathcal{C}(A,B)|}
 \ds \sum_{E \in \mathcal{C}(A, B)}
  \ds\sum_{p \leq x \atop p \nmid \Delta(E)} 
  \tau(d_{1,p}(E)) 
  = 
  C_{\tau(d_1)} \li(x) + \O\left( \frac{x}{(\log{x})^c} \right), 
\end{equation}
\begin{equation} \label{d2onaverage}
\frac{1}{|\mathcal{C}(A,B)|} 
\ds \sum_{E \in \mathcal{C}(A, B)}
\ds \sum_{p \leq x \atop p \nmid \Delta(E)} 
d_{2,p}(E) 
= 
\frac{1}{2} C_{d_2} \li(x^2) + \O\left( \frac{x^2}{(\log{x})^c} \right).
\end{equation}
These results  confirm Conjectures \ref{tau-d-one} - \ref{d-two} on average.

Conjecture \ref{d-one} is also expected to hold on average;
that is,
for suitably large 
$A = A(x)$, $B = B(x)$, we expect
\begin{equation} \label{d1onaverage}
\frac{1}{|\mathcal{C}(A,B)|} 
\ds\sum_{E \in \mathcal{C}(A, B)  \atop{
 E \; \text{without CM}
 }}
 \ds\sum_{
 p \leq x 
 \atop{
 p \nmid \Delta(E)
 }
 } 
 d_{1,p}(E) 
 = 
 C_{d_1} \li(x) + \o \left( \frac{x}{\log{x}} \right).
\end{equation}
While this average is open, Akbary and Felix proved related results supporting it (see \cite[Remark 1.7]{AkFe}).

In this  paper we investigate the constants
$C_{d_1, \text{non-CM}}(E)$, $C_{d_1, \text{CM}}(E)$, $C_{\tau(d_1)}(E)$, and $C_{d_2}(E)$,
on their own 
{\it{and}}
in relation to the idealized constants
$C_{d_1}$, $C_{\tau(d_1)}$, and $C_{d_2}$.
Specifically,  using properties of the division fields of $E/\Q$, 
which we  derive from the celebrated open image theorems for elliptic curves with complex multiplication (due to Weil), 
respectively without complex multiplication (due to Serre),
we prove:

\begin{theorem}\label{constant-nonSerre}
Let $E/\Q$ be an elliptic curve.
\begin{enumerate}
\item[(i)]
Assume that $\End_{\overline{\Q}}(E) \not\simeq \Z$.
Then
$$
C_{d_1, \text{CM}}(E) \ll 1
$$
and
$$
C_{d_2}(E)
 \leq
C_{\tau(d_1)}(E)  
\ll 1.
$$
\item[(ii)]
Assume that $\End_{\overline{\Q}}(E) \simeq \Z$.
Let $Y^2 = X^3 + aX +b$ be a Weierstrass model  for $E$ with $a,b \in \Z$.
There exists a positive absolute constant $\gamma$ such that
$$
C_{d_1, \text{non-CM}}(E)
\ll_{\varepsilon} 
\left(\max\left\{|a|^3, |b|^2\right\} 
 \log\left(\max\left\{|a|^3, |b|^2\right\}\right)^\gamma\right)^{\varepsilon}
 \; \; \; \text{for any} \; \varepsilon > 0
$$
and
$$
C_{d_2}(E)
 \leq 
 C_{\tau(d_1)}(E) 
 \ll
 \log \log 
 \left\{
 \max\left\{|a|^3,|b|^2\right\} 
 \log\left(\max\left\{|a|^3,|b|^2\right\}\right)^\gamma
 \right\}.
$$
\end{enumerate}
The $\ll$-constants are absolute, while the $\ll_{\varepsilon}$-constant depends on $\varepsilon$.
\end{theorem}

Furthermore, by narrowing down our focus to generic elliptic curves, we prove explicit formulae for our Titchmarsh divisor constants:
\begin{theorem}\label{serrecurveconstant}
Let $E/\Q$ be a Serre curve, that is, an elliptic curve over $\Q$ whose adelic Galois representation has maximal image.
Let 
\begin{equation}\label{mESerre}
 m_E =
 \left\{
 \begin{array}{cc}
 2  \left| \Delta_{\text{sf}}(E) \right| & \;  \text{if $\Delta_{\text{sf}}(E) \equiv 1 (\mod 4)$,}
 \\
 4  \left|\Delta_{\text{sf}}(E) \right|  & \; \text{otherwise},
 \end{array}
 \right.
 \end{equation}
 \noindent
 where $\Delta_{\text{sf}}(E)$  denotes the squarefree part of any Weierstrass discriminant of $E$.
Then 
\begin{eqnarray}
C_{d_1, non-CM}(E)
&=&
 C_{d_1} 
 \left( 
 1 +  \frac{1}{m_E^3} \ds\prod_{\ell \mid m_E} \frac{1}{(1-\ell^{-2})(1-\ell^{-3})+\ell^{-3}}
 \right),
 \label{serre-d-one}
 \\
 C_{\tau(d_1)}(E) 
&=&  
C_{\tau(d_1)}
\left( 
1 + \frac{1}{m_E^4} \ds\prod_{\ell \mid m_E} \frac{1}{(1-\ell^{-1})(1-\ell^{-2})(1-\ell^{-4}) + \ell^{-4}} 
\right),
\label{serre-tau-d-one}
\\
C_{d_2} (E) 
&=&
 C_{d_2} 
 \left( 
 1 + \frac{(-1)^{\omega(m_E)}}{m_E^4} \ds\prod_{\ell \mid m_E} \frac{1}{(1-\ell^{-2})(1-\ell^{-5}) - \ell^{-4}} 
 \right).
 \label{serre-d-two}
\end{eqnarray}
\end{theorem}

Finally, we use the above two results to prove that the average of  the individual constants 
gives rise to the idealized constant:

\begin{theorem}\label{main-thm}
For any $A(x), B(x) > 2$, tending to infinity with $x$ such that
the ratios of their logarithms remain bounded, we have
\begin{equation}\label{constant-average}
\lim_{x \to \infty} \ \frac{1}{|{\mathcal{C}}(A(x), B(x))|}
\ds\sum_{
E \in {\mathcal{C}}(A(x), B(x))
}
C(E)
=
C.
\end{equation}
Here, the pair $(C(E), C)$ is, respectively,
$(C_{d_1}(E), C_{d_1})$, 
$(C_{\tau(d_1)}(E), C_{\tau(d_1}))$,
and
$(C_{d_2}(E), C_{d_2})$,
with
$C_{d_1}(E)$ denoting
$C_{d_1, \text{non-CM}}(E)$ if $E$ is without complex multiplication, 
and 
$C_{d_1, \text{CM}}(E)$ if $E$ is with complex multiplication.
\end{theorem}

\begin{remark}\label{intro-remark}
{\emph{
\begin{enumerate}
\item[(i)]
The constant $\gamma$ occurring in Theorem \ref{constant-nonSerre} 
  is known as the 
Masser-W\"{u}stholz constant and originates in \cite{MaWu}; it has been studied computationally in  \cite{Ka}.
\item[(ii)]
The constant $m_E$ occurring in Theorem \ref{serrecurveconstant} was first introduced in \cite{Jo10} in relation to Serre's open image theorem from \cite{Se72}. In Section 2,  we will revisit its original definition (see Theorem \ref{open-nonCM}) and we will confirm that  it satisfies equation (\ref{mESerre}) above (see Proposition \ref{mESerreProp}).
\item[(iii)]
In Section 5, we will actually prove a stronger result than (\ref{constant-average}) by bounding, from above, 
$$
\frac{1}{|{\mathcal{C}}(A, B)|}
\ds\sum_{E \in {\mathcal{C}}(A, B)}
\left|C(E) - C\right|^n
$$ 
for any  integer $n \geq 1$ and for any $A, B > 2$;
see equations (\ref{averageCM}), (\ref{averagenonCMd2}),  (\ref{averagenonCMd1}), and (\ref{averageSerre}).
\item[(iv)]
It is a difficult problem to calculate the constants 
$C_{d_1, \text{non-CM}}(E)$, 
$C_{d_1, \text{CM}}(E)$, 
$C_{\tau(d_1)}(E)$, and $C_{d_2}(E)$
for an arbitrary elliptic curve $E/\Q$.
The underlying reason is the difficulty in attaining a complete understanding, explicit in terms of $E$, 
of the non-trivial intersections between the division fields of $E$.
This may be rephrased as a question about understanding the constant $m_E$, alluded to in Theorem \ref{main-thm}
and introduced in the upcoming 
Theorem \ref{open-CM} and Theorem \ref{open-nonCM} of Section 2.
In Theorem \ref{constant-nonSerre} we succeed in bounding the constants
$C_{d_1, \text{non-CM}}(E)$, 
$C_{d_1, \text{CM}}(E)$, 
$C_{\tau(d_1)}(E)$, and $C_{d_2}(E)$
by 
either working with all twists of $E/\Q$ in the complex multiplication case, 
or by using an upper bound for $m_E$ in the other case.
\item[(v)]
 Theorem \ref{serrecurveconstant}
shows 
that
the constants 
$C_{d_1, \text{non-CM}}(E)$, 
$C_{\tau(d_1)}(E)$, and $C_{d_2}(E)$
of 
Conjectures \ref{d-one} -  \ref{d-two}
are never equal to the idealized constants
$C_{d_1}$, 
$C_{\tau(d_1)}$, and $C_{d_2}$.
This is in contrast to 
other questions about reductions of elliptic curves, such as Koblitz's Conjecture about the primality of 
$|\overline{E}(\F_p)|$.
Moreover, the explicit formulae of Theorem \ref{serrecurveconstant}
can be used to study Conjectures \ref{d-one} -  \ref{d-two} numerically, as done in \cite{CoFiInYi}.
Combining these formulae with Theorem \ref{main-thm},
we may also naturally expect a positive bias in the remainder term associated to an (appropriately formulated) average version of these conjectures; see Section 6.
\item[(vi)]
Theorem \ref{main-thm}
contributes
to the research on averages of constants arising in the study of {\it{reductions}} of elliptic curves over $\mbq$, as pursued in
 \cite{AkDaHaTh}, \cite{BaCoDa}, \cite{BaSh}, \cite{CoIwJo}, and \cite{Jo09}.
It also  complements
research on averages of constants arising in the study of {\it{all}} elliptic curves over the field $\mathbb{F}_p$,
as pursued in \cite{DaKoSm}, \cite{Ge}, \cite{Ho}, \cite{KaPe}, and \cite{Vl}.  
The connection between the former ``global'' viewpoint and the latter ``local'' viewpoint involves the question of to what extent the reductions of a fixed elliptic curve $E/\Q$  behave like random elliptic curves over $\mathbb{F}_p$.
In our average approach we follow the methods of  \cite{Jo09}
and realize the global-to-local connection via
Theorem \ref{serrecurveconstant} 
and 
Jones' result that most elliptic curves over $\Q$ are Serre curves \cite{Jo10}.
\item[(vii)] A natural question is whether a conjecture similar to Conjecture \ref{tau-d-one}  may be formulated regarding the behaviour of 
$\tau(d_{2, p})$. Since we have not found such an investigation in the literature, we relegate it to future research.
\end{enumerate}
}}
\end{remark}

\bigskip

\noindent
{\bf{Notation.}}  
Throughout the paper, we follow the following standard notation. 

\noindent
$\bullet$
The letters $p$ and $\ell$ denote rational primes. The letters $d, k, m, n$ denote rational integers.
The letters $\phi$, $\tau$, $\omega$ denote the Euler function, the divisor function, and the prime factor counting function.
For an integer $m$, $|m|$ denotes its absolute value, $\rad(m)$ its radical, and $v_{\ell}(m)$ its valuation at a prime $\ell$.
For nonzero integers $m$, $n$,
we write
$m \mid n^{\infty}$ for the assumption that every prime divisor of $m$ also divides $n$. The symbol $\N$ stands for all natural numbers, including $0$.

%




\noindent
$\bullet$
For two functions $f, g: D \longrightarrow \R$, with $D \subseteq \C$ and $g$ positive, we write $f (x)= \O(g(x))$ or $f(x) \ll g(x)$ if there is a positive constant $c_1$ such that $|f(x)| \leq c_1 g(x)$ for all $x \in D$.
  If $c_1$ depends on another specified  constant $c_2$, we may write $f(x) = \O_{c_2}(g(x))$ or $f(x) \ll_{c_2} g(x)$. 
  If  $c := \ds\lim_{x \rightarrow \infty} \frac{f(x)}{g(x)}$ exists, 
  we write $f(x) \sim c  \ g(x)$. 

\noindent
$\bullet$
 For a field $K$, we write $\overline{K}$ for a fixed  algebraic closure and $G_K$ for the absolute Galois group $\Gal\left(\overline{K}/K\right)$.

\bigskip
\section{Generalities about elliptic curves}

\bigskip

In this section, we review the main properties of elliptic curves needed in the proofs of our main results. While many of these properties are standard 
(Theorem \ref{ECM} - Corollary \ref{mE-nonCM}), 
the ones towards the end of the section (Lemma \ref{vertical-division} - Theorem \ref{serre-curves}) are less known, but crucial to our approach.

\subsection{General notation}

For an elliptic curve $E/\Q$, we use the following notation.
We denote by $j = j(E)$ the $j$-invariant of $E$.
We denote by 
$\End_{\overline{\Q}}(E)$  and $\Aut_{\overline{\Q}}(E)$
the endomorphism ring
and 
the automorphism ring
of $E$
 over $\overline{\Q}$.
We denote by $E(\Q)$ and $E(\overline{\Q})$ the groups of $\Q$-rational and $\overline{\Q}$-rational points of $E$, and by
$E(\Q)_{\text{tors}}$ and $E_{\text{tors}} := E(\ol{\mbq})_{\text{tors}}$ their respective torsion subgroups. 

For an integer $m \geq 1$,
we denote by $E[m] \colonequals E(\overline{\Q})[m]$ the  group of $m$-division points of $E(\overline{\Q})$.  
This has the structure of a free $\Z/m\Z$-module of rank $2$, with a  $\Z/m \Z$-linear action of the absolute Galois group $G_{\Q}$.  
Thus, fixing an isomorphism $(\Z/m\Z)^2 \to E[m]$, which amounts to choosing a $\Z/m \Z$-basis for $E[m]$, 
we obtain a Galois representation
\begin{equation*} \label{mandErep}
\varphi_{E,m} : G_\Q \longrightarrow \GL_2(\Z/m\Z).
\end{equation*}
Note that the $m$-division field $\Q(E[m])$ of $E$, defined by adjoining to $\Q$ the $x$ and $y$ coordinates of the points in $E[m]$, satisfies
$$\Q(E[m]) = \overline{\Q}^{\Ker{\varphi_{E,m}}} \; \text{and} \;  [\Q(E[m]):\Q] = \left| \Im(\varphi_{E,m}) \right|.$$

Choosing compatible bases for all $E[m]$, we form the inverse limit over $m$ ordered by divisibility
and obtain a continuous Galois representation 
$$
\varphi_E :  G_\Q \longrightarrow \GL_2(\hat{\Z}),
$$
where $\hat{\Z} :=  \ds\varprojlim_{m} \Z/m\Z$.
By the Chinese Remainder Theorem, we have $\hat{\Z} \simeq \ds\prod_{\ell} \Z_\ell$,
and,  post-composing with projections,  we obtain the Galois representations
$$
\varphi_{E, m^{\infty}} : G_{\Q} \longrightarrow \GL_2(\Z_m),
$$
where $\Z_m \colonequals \ds\prod_{\ell \mid m} \Z_\ell$.

\subsection{Morphisms}

\begin{theorem}\label{ECM}(\cite[Thm 7.30]{Cox})
Let $E/\Q$ be an elliptic curve. The following statements hold.
\begin{enumerate}
\item[(i)]
Either 
$\End_{\overline{\Q}}(E) \simeq \Z$, in which case we say that $E$ is {\bf{without complex multiplication}},
or
$\End_{\overline{\Q}}(E)$ is an order $O$ in an imaginary quadratic field $K$, in which case we say that $E$ is {\bf{with complex multiplication by $K$}}.
\item[(ii)]
If 
$\End_{\overline{\Q}}(E) \simeq O \not\simeq \Z$ and $K$ is the field of fractions of $O$,
then
 $O$ and the ring of integers of $K$ have class number 1.
In particular, by the Baker-Heegner-Stark Theorem,
$K$ is one of nine possible fields,
$$
\Q(i), 
\Q\left(\sqrt{-2}\right), 
\Q\left(\sqrt{-3}\right), 
\Q\left(\sqrt{-7}\right), 
\Q\left(\sqrt{-11}\right), 
\Q\left(\sqrt{-19}\right), 
\Q\left(\sqrt{-43}\right),
\Q\left(\sqrt{-67}\right), 
\Q\left(\sqrt{-163}\right),
$$
while
 $O$ is one of 
thirteen possible orders, of discriminant
$$
-3, -4, -7, -8, -11, -12, -16, -19, -27, -28, -43, -67, -163.
$$
\end{enumerate}
\end{theorem}

\begin{corollary}(\cite[p. 261]{Cox})
\label{13j}

\noindent
Let $E/\Q$ be an elliptic curve.
If $\End_{\overline{\Q}}(E)  \not\simeq \Z$,
then $j(E)$ is one of thirteen possible integers,
$$
0,
2^6 \cdot 3^3,
- 3^3 \cdot 5^3,
2^6 \cdot 5^3,
- 2^{15},
2^4 \cdot 3^3 \cdot 5^3,
2^3 \cdot 3^3 \cdot 11^3,
$$
$$
- 2^{15} \cdot 3^3,
- 2^{15} \cdot 3 \cdot 5^3,
3^3 \cdot 5^3 \cdot 17^3,
- 2^{18} \cdot 3^3 \cdot 5^3,
- 2^{15} \cdot 3^3 \cdot 5^3 \cdot 11^3,
- 2^{18} \cdot 3^3 \cdot 5^3 \cdot 23^3 \cdot 29^3,
$$
listed in the same order as the above list of possible discriminants.
 \end{corollary}
 
 Note that two elliptic curves of the same $j$-invariant, while isomorphic over $\overline{\Q}$, may fail to be isomorphic over $\Q$; 
 we call such curves {\bf{twists}} of each other. 
 
\begin{proposition}\label{cohom}(\cite[Chapter 3, Cor 10.2]{Si})

\noindent
Let $E/\Q$ be an elliptic curve.
Then the set of all twists of $E/\Q$ is classified by the Galois cohomology group 
\[H^1(G_{\Q}, \Aut_{\overline{\Q}}(E)) \simeq \begin{cases} \Q^\times / \left(\Q^\times\right)^2  &\text{ if } j(E) \neq 0, 1728, \\ \Q^\times / \left(\Q^\times\right)^4 &\text{ if } j(E) = 1728, \\ \Q^\times / \left(\Q^\times\right)^6 &\text{ if } j(E) = 0.\end{cases} \]
\end{proposition}

\begin{corollary}\label{twists}
Let $E/\Q$, $E'/\Q$ be elliptic curves with $j(E) = j(E')$. Then there exists an extension $L/\Q$ with 
$[L : \Q] \leq |\Aut_{\overline{\Q}}(E)| \leq 6$ such that $E \simeq_L E'$.
\end{corollary}

\subsection{Open image theorems}

\begin{theorem}\label{open-CM} (\cite{We55}, \cite{We55bis})

\noindent
Let $E/\Q$ be an elliptic curve such that $\End_{\overline{\Q}}(E) \simeq O \not\simeq \Z$  is an order in the imaginary quadratic field $K$.
Consider the projective limit $\widehat{O} := \ds\varprojlim_{m} O/m O$.
Then:
\begin{enumerate}
\item[(i)]
$\widehat{O}$ is a free $\hat{\mbz}$-module of rank $2$; consequently, after choosing a $\hat{\mbz}$-basis of $\widehat{O}$, we obtain an embedding
\begin{equation} \label{iotaembedding}
\iota : \widehat{O}^\times \hookrightarrow \GL_2(\hat{\mbz});
\end{equation}
\item[(ii)]
$K(E_{\text{tors}}) = \Q(E_{\text{tors}})$ is a free $\widehat{O}$-module of rank 1, acted on $\widehat{O}$-linearly by $G_K$;
\item[(iii)]
 $\varphi_E(G_K)$ is contained in $\iota(\widehat{O}^\times)$, while $\varphi_E(G_\mbq)$ is contained in the normalizer in $\GL_2(\hat{\mbz})$ of 
 $\iota(\widehat{O}^\times)$;
\item[(iv)]
upon identifying $\widehat{O}^\times$ with $\iota(\widehat{O}^\times)$, its image in $\GL_2(\hat{\mbz})$, 
the representation
\begin{equation}
{\varphi_E}|_{G_K} :  G_K \longrightarrow \widehat{O}^\times
\end{equation}
has open image in $\widehat{O}^{\times}$, that is,
$$
\left|
\widehat{O}^{\times} : {\varphi_E}(G_K)
\right|
< \infty.
$$
In particular,  there exists a smallest integer $m_E \geq 1$ such that 
$$
{\varphi_E} (G_K)
 \simeq
 \pr^{-1} \left(
 \varphi_{E, m_E}(G_K)
 \right),
$$
where 
$
\pr : \widehat{O}^{\times} \longrightarrow (O/m_E \; O)^{\times}
$
is the natural projection.
\end{enumerate}
\end{theorem}

\begin{corollary}\label{mE-CM}
Let $E/\Q$ be an elliptic curve such that $\End_{\overline{\Q}}(E) \simeq O \not\simeq \Z$.
With notation as above, 
for any integer $m \geq 1$, written  uniquely as $m = m_1 m_2$ 
for some integers $m_1$, $m_2$ with 
$m_1 \mid m_E^{\infty}$
and
$\gcd(m_2, m_E) = 1$,
we have
$$
\varphi_{E, m}(G_K)
\simeq 
\pr_{m_1,d}^{-1}(\varphi_{E, d}(G_K)) \times (O/m_2 O)^\times,
$$ 
where 
$d := \gcd(m_1,m_E)$
and
 $\pr_{m_1,d} : (O/m_1 \; O)^\times \longrightarrow (O/d \; O)^\times$ is the natural projection.
\end{corollary}

\begin{theorem}\label{open-nonCM}(\cite{Se72})

\noindent
Let $E/\Q$ be an elliptic curve such that $\End_{\overline{\Q}}(E) \simeq \Z$.
Then
$\varphi_E$ has open image in $\GL_2\left( \hat{\mbz} \right)$, that is,
$$
\left|\GL_2\left(\hat{\Z}\right) :  \varphi_E (G_{\Q}) \right| < \infty.
$$ 
In particular, there exists a smallest integer $m_E \geq 1$ such that 
$$
\varphi_E(G_{\Q})
=
\pr^{-1}(\varphi_{E, m_E}(G_{\Q})),
$$
where $\pr: \GL_2\left(\hat{\Z}\right) \longrightarrow \GL_2(\Z/m_E \Z)$ is the natural projection.
\end{theorem}
\begin{corollary}\label{mE-nonCM}
Let $E/\Q$ be an elliptic curve such that $\End_{\overline{\Q}}(E) \simeq \Z$.
With notation as above, 
for any integer $m \geq 1$, written  uniquely as $m = m_1 m_2$ 
for some integers $m_1$, $m_2$ with 
$m_1 | m_E^{\infty}$
and
$\gcd(m_2, m_E) = 1$,
we have
$$
\varphi_{E, m}(G_{\Q})
\simeq 
\pr_{m_1,d}^{-1}( \varphi_{E, d}(G_{\Q})) \times \GL_2(\Z/m_2 \Z),
$$ 
 where 
 $d := \gcd(m_1,m_E)$
 and $\pr_{m_1,d} : \GL_2(\mbz/m_1\mbz) \longrightarrow \GL_2(\mbz/d\mbz)$ is the natural projection.
 \end{corollary}

A useful consequence to the above two open image theorems is:

\begin{lemma}\label{vertical-division}
Let $E/\Q$ be an elliptic curve and let $m_E$ be as in Theorem \ref{open-CM}, respectively Theorem \ref{open-nonCM}.
Let $\ell \mid m_E$ and $d \mid m_E$ with $\ell \nmid d$ (recall that $\ell$ denotes a rational prime). 
\begin{enumerate}
\item[(i)]
If $\End_{\overline{\Q}}(E) \not\simeq \Z$, then
$$
\left[K\left(E\left[\ell^{v_\ell(m_E) + \delta} d \right]\right) : K\right] = 
\ell^{2\delta} 
\left[K\left(E\left[\ell^{v_\ell(m_E)} d\right]\right) : K\right]
\; \; \; \; \forall \; \delta \in \N,
$$
where $K \simeq \End_{\ol{\Q}}(E) \otimes \Q$.
\item[(ii)]
If $\End_{\overline{\Q}}(E) \simeq \Z$, then
$$
\left[\Q\left(E\left[\ell^{v_\ell(m_E) + \delta} d\right]\right) : \Q\right] = 
\ell^{4 \delta}
\left[\Q\left(E\left[\ell^{v_\ell(m_E)} d\right]\right) : \Q\right]
\; \; \; \; \forall \; \delta \in \N.
$$
\end{enumerate}
\end{lemma}

\begin{proof}
We apply Corollary \ref{mE-CM}, respectively Corollary \ref{mE-nonCM}, with 
$m := \ell^{v_\ell(m_E) + \delta} d$, 
in which case
$m_1 = m$
and 
 $m_2 = 1$,
 giving
$\gcd(m_1,m_E) = \ell^{v_\ell(m_E)}d$.  
We also observe that, since $v_\ell(m_E) > 0$, 
\[
\begin{split}
\#\Ker\left( \left(O/\ell^{v_\ell(m_E) + \delta} d \; O\right)^\times \longrightarrow \left(O/\ell^{v_\ell(m_E)} d \; O\right)^\times \right)  
&= \ell^{2\gd}, \\
\#\Ker\left( \GL_2\left(\mbz/ \ell^{v_\ell(m_E) + \delta} d \mbz\right)\rightarrow \GL_2\left(\mbz/ \ell^{v_\ell(m_E)} d \mbz\right) \right)  &= \ell^{4\gd}.
\end{split}
\]
Lemma \ref{vertical-division}  follows.

\end{proof}

%

In applications, 
 it  is desirable to explicitly bound $m_E$ in terms of $E$; such a bound is  illustrated in the result below:

\begin{proposition} \label{mEboundlemma}
Let $E/\Q$ be an elliptic curve such that $\End_{\overline{\Q}}(E) \simeq \Z$.
Let $m_E$ be as in  Theorem \ref{open-nonCM},
let $Y^2 = X^3 + aX +b$ be a Weierstrass model  for $E$ with $a,b \in \Z$,
and
denote by $\Delta(a, b) := -16 (4 a^3 + 27 b^2)$ the discriminant of this model.
Then there exists a positive absolute constant $\gamma$ such that
$$
m_E  \leq 2 \ |\GL_2(\hat{\Z}) : \varphi_{E}(G_\Q)|  \ \rad(|\Delta(a, b)|) \ll \max\{|a|^3,|b|^2\} \log(\max\{|a|^3,|b|^2\})^\gamma,
$$
where the $\ll$-constant is  absolute.
\end{proposition}

\begin{proof}

The bound 
\[
m_E  \leq 2 \ |\GL_2(\hat{\Z}) : \varphi_{E}(G_\Q)| \ \rad(|\Delta(a, b)|)
\]
follows from the main result in \cite{Jo17},
while the bound
\[
|\GL_2(\hat{\Z}) : \varphi_{E}(G_\Q)| \ll  \log(\max\{|a|^3,|b|^2\})^\gamma
\]
follows from \cite[Theorem 1.1]{Zy} (see also \cite{MaWu}).  The bound $\rad(|\Delta(a, b)|) \ll \max\{|a|^3,|b|^2\}$ is straightforward.
\end{proof}

\subsection{Serre curves}

\begin{lemma}\label{serre-subgroup}(\cite[Section 5.5]{Se72})

\noindent
Let $E/\Q$ be an elliptic curve. 
Then 
$
\left|
\GL_2\left(\hat{\Z}\right) : \varphi_E(G_{\Q})
\right|
\geq 2.
$
\end{lemma}

\noindent
In particular, no elliptic curve $E/\Q$ satisfies
$
\left|
\GL_2(\Z/m \Z) : 
\varphi_{E, m}(G_{\Q})
\right|
= 
1
$
for all integers $m \geq 1$.  Rather, the best we can hope for is  captured in the following definition:

\begin{definition} \label{Serrecurvedef}
An elliptic curve $E/\Q$ is called a {\bf{Serre curve}} if
$$
\left|
\GL_2(\hat{\Z}) : 
\varphi_{E}(G_{\Q})
\right|
= 2.
$$
\end{definition}

\noindent Equivalently, an elliptic curve is a Serre curve if and only if
\begin{equation} \label{serrecurvedefatfinitelevel}
\left| \GL_2(\mbz/m\mbz) : \varphi_{E,m}(G_\Q) \right| \leq 2 \quad \forall m \geq 1.
\end{equation}

Given $E/\Q$ and denoting by $\Delta_{\text{sf}}(E)$ the squarefree part of the discriminant $\Delta(E)$ of any Weierstrass model for $E$,
the bound in Lemma \ref{serre-subgroup} arises from the containments
\begin{equation} \label{Serrecontainments}
\Q\left(\sqrt{\gD(E)}\right) \subseteq \Q\left(E[2]\right), \quad\quad \Q\left(\sqrt{\gD(E)}\right) \subseteq \Q\left(\zeta_{|d_E|}\right) \subseteq \Q\left(E[|d_E|]\right),
\end{equation}
where 
\[
d_E := \disc\left( \Q\left(\sqrt{\gD(E)}\right) \right) = 
\left\{
\begin{array}{cc}
 \gD_{\text{sf}}(E)  & \text{ if } \gD_{\text{sf}}(E) \equiv 1 (\mod 4),
 \\
4  \gD_{\text{sf}}(E)  & \text{ otherwise.}
\end{array}
\right.
\]
The existence of an integer $d_E$ satisfying \eqref{Serrecontainments} is guaranteed by the Kronecker-Weber Theorem, since $\Q\left(\sqrt{\gD(E)}\right)$ is abelian over $\Q$; this value of $d_E$ minimizes $|d_E|$ subject to \eqref{Serrecontainments}.)
It follows that 
\begin{equation} \label{serresubgroupexplicit}
\varphi_E(G_\mbq) \subseteq \{ g \in \GL_2(\hat{\mbz}) : \epsilon(g) = \chi_E(g) \},
\end{equation}
where the two maps
\[
\begin{split}
\epsilon : &\GL_2(\hat{\mbz}) \rightarrow \GL_2(\mbz/2\mbz) \simeq S_3 \rightarrow \{ \pm 1 \}, \\
\chi_E : &\GL_2(\hat{\mbz}) \rightarrow \hat{\mbz}^\times \rightarrow (\mbz/|d_E|\mbz)^\times \rightarrow \{ \pm 1 \}
\end{split}
\]
are defined as follows:
$\epsilon$ is the projection modulo $2$ followed by the signature character on the permutation group $S_3$ (which is also the unique non-trivial multiplicative character on $\GL_2(\mbz/2\mbz)$);
 $\chi_E$ is the determinant map, followed by the reduction modulo 
$|d_E|$, and then followed by the Kronecker symbol $\left( \frac{d_E}{\cdot} \right)$.


\begin{proposition}\label{mESerreProp}
Let $E/\Q$ be a Serre curve and let $\Delta(E)$ be the discriminant of any
 Weierstrass model for $E$.
Then:
\begin{enumerate}
\item[(i)]
$\End_{\ol{\Q}}(E) \simeq \Z$;
\item[(ii)]
 $E(\Q)_{\text{tors}}$ is trivial;
 \item[(iii)]
 the integer $m_E$ introduced in Theorem \ref{open-nonCM} satisfies
 \begin{equation}
 m_E =
 \left\{
 \begin{array}{cc}
 2  \left| \Delta_{\text{sf}}(E) \right| & \;  \text{if $\Delta_{\text{sf}}(E) \equiv 1 (\mod 4)$,}
 \\
 4  \left|\Delta_{\text{sf}}(E) \right|  & \; \text{otherwise},
 \end{array}
 \right.
 \end{equation}
 \noindent
 where $\Delta_{\text{sf}}(E)$  denotes the squarefree part of $\Delta(E)$;
 \item[(iv)]
 for any integer $m \geq 1$,
 \begin{equation}\label{serreindex} 
[\Q(E[m]) : \Q] = 
\left\{
 \begin{array}{cc}
|\GL_2(\Z/m \Z)|  & \;  \text{if $m_E \nmid m$,}
 \\
 \frac{1}{2} |\GL_2(\Z/m\Z)| & \; \text{otherwise}.
 \end{array}
 \right.
\end{equation}
 \end{enumerate}
\end{proposition}

\begin{proof}
(i)  We proceed by contradiction. Suppose that $\End_{\overline{\Q}}(E) \simeq O \not\simeq \Z$ and let $K$ be the field of fractions of $O$. 
Then there exists an element $a \in O \backslash \mbz$ 
and there exists a rational prime $\ell$ such that the characteristic polynomial of  the action of $a$ on $E[\ell]$ has two distinct roots modulo $\ell$. Indeed, letting $\gD_a$ denote the discriminant of this characteristic polynomial, one may take any $\ell$ satisfying $\left( \frac{\gD_a}{\ell} \right) = 1$. Then
$G_K$ preserves the two eigenspaces of $a$, which implies that, written relative to an eigenbasis, we have
\begin{equation*} 
\varphi_{E, \ell}(G_K) \leq \left\{ \begin{pmatrix} * & 0 \\ 0 & * \end{pmatrix} \right\} <  \GL_2(\mbz/\ell\mbz).
\end{equation*}
We thus have $\left| \GL_2(\mbz/\ell\mbz) : \varphi_{E, \ell}(G_{\Q}) \right| \geq \ell(\ell+1) > 2$, contradicting \eqref{serrecurvedefatfinitelevel}.





(ii) We proceed by contradiction.
Suppose that $E(\Q)_{\text{tors}}$ is non-trivial. Then there exists a non-trivial  point $P \in E(\Q)[\ell]$ for some  rational prime $\ell$.
Extending the set $\{ P \}$ to an ordered $\mbz/\ell\mbz$-basis $\{ P, Q \}$ of $E[\ell]$ and writing linear transformations relative to this basis, we obtain
$$
\varphi_{E, \ell}(G_{\Q}) \leq \left\{ \begin{pmatrix} 1 & * \\ 0 & * \end{pmatrix} \right\} < \GL_2(\mbz/\ell\mbz).
$$
We thus have that $\left|\GL_2(\Z/\ell \Z) : \varphi_{E, \ell}(G_{\Q})\right| \geq (\ell + 1)(\ell - 1) > 2$,  contradicting \eqref{serrecurvedefatfinitelevel}.


(iii) Since the subgroup of $\GL_2(\hat{\Z})$ where $\chi_E$ and $\epsilon$ agree is already of index $2$, and since $E$ is a Serre curve,
 we must have equality in \eqref{serresubgroupexplicit}. 
The subgroup defined therein is determined by its image at level
\[
m_E = \lcm(2, | d_E | ) = \begin{cases}
2 | \gD_{\text{sf}}(E) | & \text{ if } \gD_{\text{sf}}(E) \equiv 1 (\mod4), \\
4 | \gD_{\text{sf}}(E) | & \text{ otherwise,}
\end{cases}
\] 
verifying \eqref{mESerre}.


(iv) Let $d \mid m_E$ and
denote by  $\pr_{m_E,d} : \GL_2(\Z/m_E\Z) \longrightarrow \GL_2 (\Z/d \Z)$ the canonical projection. 
Since 
$\left|\GL_2(\Z/m_E \Z) : \varphi_{E, m_E}(G_{\Q})\right| = 2$,
 it follows from the 
minimality of $m_E$ that 
$$
\varphi_{E,d}(G_\mbq)
=
\pr_{m_E,d}
\left(
\varphi_{E,m_E}(G_\mbq)
\right)
=
\left\{
\begin{array}{cc}
\text{index $2$ subgroup of $\GL_2(\Z/d\Z)$}
& \text{if $d = m_E$},
\\
\GL_2(\Z/m \Z)
& \text{if $d < m_E$}.
\end{array}
\right.
$$
 By Corollary \ref{mE-nonCM}, this proves \eqref{serreindex}.
\end{proof}

\subsection{Two-parameter families of elliptic curves}


\begin{lemma}\label{lemma-numcmcurves}
Let $A, B > 2$ and
consider the family 
${\mathcal{C}}(A, B)$
of $\Q$-isomorphism classes of elliptic curves 
$E = E_{a, b}$ defined by
$Y^2 = X^3 + aX +b$
with
$a, b \in \Z$
and
$|a| \leq A,
|b| \leq B$. 
Then
\begin{equation*}
\frac{1}{|\mathcal{C}(A, B)| }
\ds\sum_{
E \in {\mathcal{C}}(A, B)
\atop{E \; \text{CM}}
}
1
\ll
\frac{1}{A} + \frac{1}{B}.
\end{equation*}
More precisely,
$$
\#\{E \in {\mathcal{C}}(A, B): j(E) = 0
\}
\sim
\frac{2}{\zeta(6)} B,
$$
$$
\#\{E \in {\mathcal{C}}(A, B): j(E) = 1728
\}
\sim
\frac{2}{\zeta(4)} A,
$$
and, for each of the $j$-invariants of Corollary \ref{13j} with  $j \neq 0, 1728$, 
$$
\#\{E \in {\mathcal{C}}(A, B): j(E) = j
\}
\ll_{\varepsilon}
\min\left\{
A^{\frac{1}{2} + \varepsilon},
B^{\frac{1}{3} + \varepsilon}
\right\}
$$
for any $\varepsilon > 0$.
The $\ll$-constant is absolute, while the $\ll_{\varepsilon}$-constant depends on $\varepsilon$.
\end{lemma}
\begin{proof}
We recall that associated to an elliptic curve $E_{a, b}/\Q$, and in particular to a Weierstrass equation $Y^2 = X^3 + a X + b$, we have the $j$-invariant
$
j(a, b) := 1728 \frac{4 a^3}{4a^3 + 27b^2},
$
which encodes the $\ol{\Q}$-isomorphism class of $E_{a, b}$:
two elliptic curves $E_{a, b}/\Q$, $E_{a', b'}/\Q$ are $\ol{\Q}$-isomorphic if and only if $j(a, b) = j(a', b')$;
furthermore, $E_{a, b}/\Q$, $E_{a', b'}/\Q$ are $\Q$-isomorphic
 if and only if
\begin{equation}\label{isom-abu}
\exists  \; u \in {\Q}^{\times}
\;
\text{such that}
\;
a = u^4 a'
\;
\text{and}
\;
b = u^6 b'.
\end{equation}
\noindent

In view of the above and of Corollary \ref{13j}, it suffices to estimate the cardinality of 
$$
{\mathcal{C}}_{j}(A, B) :=
\left\{
(a, b) \in \Z \times \Z:
|a| \leq A, |b| \leq B,  
\Delta(a, b) \neq 0,
\gcd\left(a^3, b^2\right)
\, \text{12th power free},
\
j(a, b) = j
\right\}
$$
for each of the 13 occurring $j$-invariants. We will consider the cases $j = 0$, $j = 1728$, and $j \neq 0, 1728$ separately.

Note that $j (a, b) = 0$ is equivalent to $a = 0$. Thus
\begin{eqnarray}\label{j0}
\left|{\mathcal{C}}_{0}(A, B)\right|
=
\#
\left\{
b \in \Z \backslash \{0\}:
b \; \text{6th power free},
 |b| \leq B
\right\}
\sim
\frac{2}{\zeta(6)} B
\end{eqnarray}
(see \cite[p.355 ]{HaWr} for a standard approach towards such asymptotics).

Similarly, note that $j (a, b) = 1728$ is equivalent to $b = 0$. Thus
\begin{eqnarray}\label{j1728}
\left|{\mathcal{C}}_{1728}(A, B)\right|
=
\#
\left\{
a \in \Z \backslash \{0\}:
a \; \text{4th power free},
 |a| \leq A
\right\}
\sim
\frac{2}{\zeta(4)} A.
\end{eqnarray}

Now let us fix $j \neq 0, 1728$. Setting
$$
c(j) := \frac{4}{27} \left(\frac{1728}{j} - 1\right) \in \Q^{\times} \backslash \{0\},
$$
$$
N_j(A) 
:=
\#\left\{
a \in \Z \backslash \{0\}:
|a| \leq A,
c(j) a^3 = \beta^2
\; \text{for some} \; \beta \in \Z \backslash \{0\}
\right\}
$$
and 
$$
N_j(B)
:=
\#\left\{
b \in \Z \backslash \{0\}:
|b| \leq B,
\frac{1}{c(j)} b^2 = \alpha^3
\; \text{for some} \; \alpha \in \Z \backslash \{0\}
\right\},
$$
and noting that
\[
1728 \frac{4a^3}{4a^3 + 27b^2} = j \; \Longleftrightarrow \; a^3 c(j) = b^2,
\]
we obtain
\begin{equation}\label{CjAB}
\left|
{\mathcal{C}}_{j}(A, B)
\right|
\leq
\min \left\{
N_j(A), N_j(B)
\right\}.
\end{equation}

To estimate $N_j(A)$, let $a_1$, $a_2$ be two distinct integers counted in the set giving rise to $N_j(A)$.
There exist exactly two rational numbers $u = u(a_1, a_2) \in \Q^{\times}$ such that
\begin{equation}\label{a1a2u}
a_1^3 = u^2 a_2^3,
\end{equation}
in particular such that $u^2 a_2^3$ is a cube in $\Q^{\times}$. Consequently, there exists a unique $v = v(u) \in \Q^{\times}$ such that
$u = v^3$, thus such that
\begin{equation}\label{a1a2v}
a_1^3 = v^6 a_2^3.
\end{equation}

Recalling that $|a_1| \leq A$ and $a_2 \in \Z \backslash \{0\}$, we deduce that
$$
|v| \leq \sqrt{\frac{A}{|a_2|}} \leq \sqrt{A}.
$$
Let us write $v = \frac{n}{d}$ for uniquely determined coprime  $n, d \in \Z \backslash \{0\}$. On one hand, the above bound on $|v|$ gives
$$
|n| \leq |d| \sqrt{A}.
$$
On the other hand, (\ref{a1a2v}) gives
$d^6 a_1^3 = n^6 a_2^3$,
hence 
$$d \mid a_2.$$
Consequently, the number of $v \in \Q^{\times}$ satisfying (\ref{a1a2u}) is
$$
\ll 
\tau(a_2) \sqrt{A} 
\ll_{\varepsilon} 
A^{\frac{1}{2} + \varepsilon}
$$
for any $\varepsilon > 0$
(see \cite[Thm. 315 p. 343]{HaWr} for the standard upper bound  $\tau(a_2) \ll_{\varepsilon} a_2^{\varepsilon}$),
which implies that
\begin{equation}\label{NjA}
N_j(A) \ll_{\varepsilon} A^{\frac{1}{2} + \varepsilon}.
\end{equation}
Similarly,
\begin{equation}\label{NjB}
N_j(B) \ll_{\varepsilon} B^{\frac{1}{3} + \varepsilon}.
\end{equation}
Thus, by (\ref{CjAB}),
$$
\left|
{\mathcal{C}}_{j}(A, B)
\right|
\ll_{\varepsilon}
\min \left\{
A^{\frac{1}{2} + \varepsilon},
B^{\frac{1}{3} + \varepsilon}
\right\}
$$
for any $\varepsilon > 0$.
Combining this with (\ref{j0}) and (\ref{j1728}) completes the proof of the lemma.
\end{proof}

\begin{theorem} \label{serre-curves} (\cite[Thm. 4]{Jo10})

\noindent
Let $A, B > 2$ and
consider the family 
${\mathcal{C}}(A, B)$
of $\Q$-isomorphism classes of elliptic curves 
$E = E_{a, b}$ defined by
$Y^2 = X^3 + aX +b$
with
$a, b \in \Z$
and
$|a| \leq A,
|b| \leq B$. 
Then there exists a positive absolute constant $\gamma'$ such that
\[ 
\frac{1}{|\mathcal{C}(A, B)| }
\ds\sum_{E \in {\mathcal{C}}(A, B) \atop E \text{ is not a Serre curve}} 
1 
\ll 
\frac{(\log \min\{A, B\})^{\gamma'}}{\sqrt{\min\{A, B\}}},
\]
where the $\ll$-constant is  absolute.
\end{theorem}

\bigskip
\section{Constants for non-Serre curves: proof of Theorem \ref{constant-nonSerre}}

In this section we prove upper bounds for  the constants appearing in Conjectures \ref{d-one} - \ref{d-two}, as stated in Theorem \ref{constant-nonSerre}.
The key ingredients 
in the proof 
are 
Corollaries \ref{mE-CM} and \ref{mE-nonCM}, Lemma \ref{vertical-division}, Proposition \ref{mEboundlemma},
 and the following lemma about arithmetic functions: 

\begin{lemma}\label{lemma-mult-2}
Let 
$f, g : \N \backslash \{0\} \longrightarrow \mbc^\times$ 
be arithmetic functions satisfying:
\begin{enumerate}
\item[(i)]
$g$ is multiplicative;
\item[(ii)]
$\ds\sum_{m \geq 1} |g(m)|$ converges.
\end{enumerate}
Assume that 
$\exists \; M \in \N \backslash \{0\}$
and
$\exists \; \kappa \in \C$ with $\Re \kappa>0 $
such that:
\begin{enumerate}
\item[(iii)]
$\forall \; m_1 \mid M^\infty$
and
$\forall m_2$ with $\gcd(m_2, M)=1$,
we have
$f(m_1 m_2) = f(m_1) g(m_2)$;
\item[(iv)]
$\forall \; d \mid M^{\infty}$,
$\forall \; \ell \mid M$ with $\ell \nmid d$,
and 
$\forall \; \delta \in \N$,
we have
$f(\ell^{v_{\ell}(M) + \delta} d) = \ell^{-\delta \kappa} f(\ell^{v_{\ell}(M)} d)$. 
\end{enumerate}
Then
\[
\ds\sum_{m\geq 1} 
|f(m)|
\leq 
\ds\prod_{\ell \mid M} \left( 1 - \frac{1}{\left|\ell^{\kappa}\right|} \right)^{-1}
\left(
\ds\sum_{m \mid M} |f(m)|
\right)
\left(
\ds\sum_{m \geq 1} |g(m)|
\right).\]
\end{lemma}

\begin{proof}
By (iii) we have the almost-product formula
$$
\ds\sum_{m\geq 1} |f(m)| =  
\left( \ds\sum_{m \mid M^\infty} |f(m)| \right)
\left( \ds\prod_{\ell \nmid M} g_{\ell} \right),
$$
where,
for each rational prime $\ell$, 
$g_\ell := \ds\sum_{r\geq 0} |g(\ell^r)| = 1 + \sum_{r \geq 1} |g(\ell^r)|$.
Since
\[
\ds\prod_{\ell \nmid M} g_{\ell} \leq \ds\prod_\ell  g_{\ell} = \ds\sum_{m \geq 1} |g(m)|,
\]
it remains to bound $\ds\sum_{m \mid M^\infty} |f(m)|$.

Let $M = \ell_1^{\alpha_1} \ldots \ell_n^{\alpha_n}$ be the prime factorization of $M$. 
For each subset $I \subseteq \{1, \ldots, n\}$, possibly empty, define
$$
M_I :=
\ds\prod_{i \in \{1, \ldots, n\}   \backslash I  }
\ell_i^{\alpha_i - 1},
$$
and for each $m \mid M^{\infty}$, whose unique prime factorization we write as 
$m = \ell_1^{\beta_1} \ldots \ell_n^{\beta_n}$
with $\beta_1 \geq 0, \ldots \beta_n \geq 0$, define
$$
I_m := \{1 \leq i \leq n: \beta_i \geq \alpha_i\}.
$$
Partitioning the integers $m \mid M^{\infty}$ according to the subsets $I_m \subseteq \{1, \ldots, n\}$
and 
using  property (iv) of $f$, we obtain
\begin{eqnarray*}
\ds\sum_{m \mid M^{\infty}} |f(m)|
&=&
\ds\sum_{I \subseteq \{1, \ldots, n\}}
\ds\sum_{
m \mid M^{\infty}
\atop{
I_m = I
}
}
|f(m)|
\\
&=&
\ds\sum_{I \subseteq \{1, \ldots, n\} }
\ds\sum_{d \mid M_I}
\ds\sum_{
(\delta_i)_{i \in I} 
\atop{
\delta_i \in  \N \; \forall i}
 }
\left|
f\left(
\ds\prod_{i \in I}
\ell_{i}^{\alpha_i + \delta_i} d
\right)
\right|
\\
&=&
\ds\sum_{I \subseteq \{1, \ldots, n\} }
\left(
\ds\sum_{d \mid M_I}
\left|
f\left(
\ds\prod \ell_i^{\alpha_i} d
\right)
\right|
\right)
\left(
\ds\prod_{i \in I}
\ds\sum_{\delta_i \in \N}
\left|
\frac{1}{\ell_i^{\kappa \delta_i}}
\right|
\right)
\\
&\leq&
\ds\prod_{\ell \mid M} \left( 1 - \frac{1}{\left| \ell^{\kappa}\right|} \right)^{-1}
\ds\sum_{m \mid  M} |f(m)|.
\end{eqnarray*}
This completes the proof.
\end{proof}


\bigskip

\noindent
{\bf{Proof of part (i) of Theorem \ref{constant-nonSerre}}}.


\noindent
We begin by using Lemma \ref{lemma-mult-2} to verify that, 
for the arbitrary elliptic curve $E/\Q$ with complex multiplication by the order $O$ in the imaginary quadratic field $K$, we have
\begin{equation} \label{sumoverKdiv}
\sum_{m \geq 1} \frac{1}{[ K(E[m]) : K ]} < \infty.
\end{equation}
By Corollary \ref{mE-CM}, for any integer $m \geq 1$, written uniquely as $m = m_1 m_2$ with 
$m_1 \mid m_E^{\infty}$
and
$(m_2, m_E) = 1$,
we have
\begin{equation}\label{CM-vertical}
[K(E[m]) : K] = [K(E[m_1]) : K] \cdot \Phi_{O}(m_2),
\end{equation}
where $\Phi_O$ denotes the Euler function on $O$, that is, $\Phi_O(n) := \left|(O/n O)^{\times}\right|$ for any positive integer $n$.
Furthermore,  
by part (i) of  Lemma \ref{vertical-division},
for any $\ell$, $d$ with
$\ell \mid m_E$,
$d \mid m_E$, 
 $\ell \nmid d$,
 and for any $\delta \in \N$, we also have
\begin{eqnarray}\label{kappa-CM}
\left[
K\left(
E\left[
\ell^{v_{\ell}(m_E) + \delta} 
\ d
\right]
\right) 
:
K\left(
E\left[
\ell^{v_{\ell}(m_E)}
\  d
\right]
\right) 
\right]
= \ell^{2 \delta}.
 \end{eqnarray}
 Defining 
 $$f(m) := \frac{1}{[K(E[m]) : K]}, \; g(m) := \frac{1}{ \Phi_{O}(m)},$$
we  want to apply Lemma \ref{lemma-mult-2} with $M= m_E$ and $\kappa = 2$.

Condition (i) follows from the Chinese Remainder Theorem.  Indeed, if $m$ and $m'$ are relatively prime positive integers, then the principal ideals $mO$ and $m'O$ are co-maximal in the ring $O$, and so, by the Chinese Remainder Theorem, we have that 
\[
\Phi_{O}(mm') := |(O/mm'O)^\times| = |(O/mO)^\times \times (O/m'O)^\times| = \Phi_{O}(m) \Phi_{O}(m'),
\]
from which it follows immediately that $g$ is multiplicative. 
Condition (ii) follows from the convergence of the series
$\ds\sum_{m\geq 1} \frac{1}{\phi(m)^2}$
since 
$$\Phi_{O}(m)\geq \phi(m)^2 \quad \forall  m \geq 1.$$
To justify the above inequality, note that 
 $$\Phi_O(\ell^r) = \ell^{2(r-1)} \Phi_O(\ell) \quad \forall r \geq 1$$
and
$$
\Phi_O(\ell)
= 
\left\{
\begin{array}{cc}
(\ell + 1)(\ell - 1)
& \text{if $\ell$ is inert in $O$},
\\
(\ell-1)^2
& \text{if $\ell$ splits in $O$},
\\
\ell (\ell - 1)
& \text{if $\ell$ ramifies in $O$},
\end{array}
\right.
$$
hence 
$$\Phi_O(\ell^r) \geq \phi(\ell^r)^2 \quad \forall  r \geq 1.$$
Conditions (iii) and (iv) follow from (\ref{CM-vertical}) - (\ref{kappa-CM}).

By  Lemma \ref{lemma-mult-2},
we obtain
\begin{eqnarray}\label{sum-over-K}
\ds\sum_{m \geq 1} \frac{1}{[K(E[m]) : K]}
&\leq&
\zeta(2)
\left(
\ds\sum_{m | m_E} \frac{1}{[K(E[m]) : K]}
\right)
\left(
\ds\sum_{m \geq 1}
\frac{1}{\Phi_{O}(m)}
\right) 
\nonumber
\\
&\leq&
\zeta(2)
\left(
\ds\sum_{m | m_E} \frac{1}{[K(E[m]) : K]}
\right)
\left(
\ds\sum_{m \geq 1}
\frac{1}{\phi(m)^2}
\right)
<
\infty.
\end{eqnarray}

Having established  \eqref{sumoverKdiv}, our goal is now to bound the sum 
$\ds\sum_{m \geq 1} \frac{1}{[K(E[m]) : K]}$
absolutely. 
To this end, let $j_1, \ldots, j_{13}$ be the $j$-invariants of  all elliptic curves over $\Q$
 with complex multiplication, as per Corollary \ref{13j}. 
For each $j_i$, choose an elliptic curve $E_i/\Q$ with $j(E_i) = j_i$ and denote by $K_i$ its complex multiplication field.
(Recalling Theorem \ref{ECM}, note that
some of the fields $K_1, \ldots, K_{13}$ occur with multiplicity since only nine of them are distinct.)  
Also note that, since (\ref{sumoverKdiv}) holds for any elliptic curve over $\Q$ with complex multiplication, it holds for each $E_i/\Q$. 
Thus we may
 define
\begin{equation}
B_{\tau(d_1)} :=
\max_{1 \leq i \leq 13}
\ds\sum_{m \geq 1}
\frac{1}{
[K_i(E_i[m]) : K_i]
},
\end{equation}
which is an absolute constant.

Now fix one of the fields $K_i$ and let $E_i'/\Q$ be an elliptic curve with $j(E_i') = j_i$.
By Corollary \ref{twists},
there exists a number field $L_i$ such that $[L_i : \Q] \leq 6$
and $E_i' \simeq_{L_i}  E_i$. Moreover, we claim that, for any integer $m \geq 1$, 
\begin{equation}\label{twist-degree}
\frac{1}{6}
[K_i(E_i[m]) : K_i]
\leq
[
L_i K_i (E_i[m]) : L_i K_i
]
\leq
[
K_i (E_i'[m]) : K_i
].
\end{equation}
Indeed, we have
\begin{eqnarray*}
[K_i(E_i[m]) : K_i]
&=&
[K_i(E_i[m]) : L_i K_i \cap K_i(E_i[m])]
\cdot
[L_i K_i \cap K_i(E_i[m]) : K_i]
\\
&=&
[L_i K_i (E_i[m]) : L_i K_i]
\cdot
[L_i K_i \cap K_i(E_i[m]) : K_i]
\\
&\leq&
[L_i K_i (E_i[m]) : L_i K_i]
\cdot
[L _i: \Q]
\\
&\leq&
6
[L_i K_i (E_i[m]) : L_i K_i]
\end{eqnarray*}
and
\begin{eqnarray*}
[L_i K_i (E_i[m]) : L_i K_i]
=
[L_i K_i (E_i'[m]) : L_i K_i]
=
[K_i (E_i'[m]) : L_i K_i \cap K_i(E_i'[m])] 
\leq
[K_i(E_i'[m]) : K_i].
\end{eqnarray*}

We deduce from (\ref{twist-degree}) that
 for any elliptic curve $E'/\Q$ with complex multiplication, say by $K'$, we have
$$
\ds\sum_{m \geq 1} \frac{1}{[K'(E'[m]) : K']}
\leq
6 B_{\tau(d_1)}.
$$
In particular, this holds for the elliptic curve $E/\Q$, and so

\begin{equation}\label{refined-sum-over-K}
\ds\sum_{m \geq 1} \frac{1}{[K(E[m]) : K]}
\leq
6 B_{\tau(d_1)}.
\end{equation}

To relate $\ds\sum_{m \geq 1} \frac{1}{[K(E[m]) : K]}$ to the constants
$C_{d_2}(E)$ and $C_{\tau(d_1)}(E)$, note that 
\[
[ K(E[m]) : K ] \leq [ \mbq(E[m]) : \mbq ]  \quad \forall  m \geq 1.
\]
Thus
\begin{equation*}
\begin{split}
 C_{d_2}(E) := 
 \ds\sum_{m \geq 1} \frac{(-1)^{\omega(m)}\phi(\rad m)}{m [\Q(E[m]):\Q]}
 &\leq 
 \ds\sum_{m \geq 1} \frac{1}{[\Q(E[m]) : \Q]}  \; \; \; \; \left(= C_{\tau(d_1)}(E) \right) \\
 &\leq
\ds\sum_{m \geq 1} \frac{1}{[K(E[m]) : K]}.
\end{split}
\end{equation*}
Using  (\ref{refined-sum-over-K}), we  deduce that the constants  $C_{d_2}(E) \leq C_{\tau(d_1)}(E) $ are absolutely bounded
by $6 B_{\tau(d_1)}$.

\bigskip


To bound 
\[
C_{d_1}(E) := C_{d_1, \text{CM}}(E) = \lim_{\gs \rightarrow 0^+} \left( \gs \sum_{m \geq 1} \frac{\phi(m)}{[ \mbq(E[m]) : \mbq ]} m^{-\gs} \right), 
\]
we proceed much as above, as follows.
Fix an arbitrary complex number $s = \sigma + it$ with $\sigma > 0$ 
and
define 
$$\ds f_s(m) := \frac{\phi(m)}{[K(E[m]) : K]}m^{-s}, \; \ds g_s(m) := \frac{\phi(m)}{ \Phi_{O}(m)}m^{-s}.$$
We want to apply Lemma \ref{lemma-mult-2} with $M= m_E$ and $\kappa_s = 1 + s$.

Condition (i) follows from the observation that a product of multiplicative functions is multiplicative.
Condition (ii)  follows from the calculation

\begin{eqnarray*}
\ds\sum_{m \geq 1} |g_s(m)|
&=&
\sum_{m \geq 1}
\frac{\phi(m)}{ \Phi_{O}(m)}m^{-\sigma}
\\
&=&
\prod_\ell
\left(
1 - \frac{\chi_{O}(\ell)}{\ell}
\right)^{-1}
\left(
1 - \frac{1}{\ell^{1 + \sigma}}
\right)^{-1}
\\
&=&
L(1,\chi_{O})\zeta(1+\sigma) < \infty,
\end{eqnarray*}
where $\chi_{O}$ is the character defined by
$$
\Phi_{O}(\ell^r) = \ell^{2r} \left(1 - \frac{1}{\ell}\right) \left(1 - \frac{\chi_{O}(\ell)}{\ell}\right)
$$
and
$$
L(1, \chi_{O}) = \ds\prod_{\ell} \left(1 - \frac{\chi_{O}(\ell)}{\ell}\right)^{-1}.
$$
Conditions (iii) and (iv) follow, as before, from (\ref{CM-vertical}) - (\ref{kappa-CM}). 

By Lemma \ref{lemma-mult-2},
we obtain 
\begin{eqnarray*}
\sum_{m \geq 1} \left| \frac{\phi(m)}{[K(E[m]) : K]}m^{-s} \right|& = & \sum_{m \geq 1} \frac{\phi(m)}{[K(E[m]) : K]}m^{-\gs}
\ds
\nonumber
\\
&\leq&
\ds\prod_{\ell \mid m_E} \left( 1 - \frac{1}{\ell^{1+\sigma}} \right)^{-1}
\left(
\ds\sum_{m \mid m_E} \frac{\phi(m)}{[K(E[m]) : K]}m^{-\sigma}
\right)
\left(
\ds\sum_{m \geq 1}
\frac{\phi(m)}{\Phi_{O}(m)}m^{-\sigma}
\right) \\
&<& \infty.
\end{eqnarray*}
Furthermore, using 
$$\ds \lim_{\gs \rightarrow 0^+} \gs \sum_{m \geq 1} \frac{\phi(m)}{\Phi_{O}(m)}m^{-\sigma} = L(1,\chi_{O}) \lim_{\gs \rightarrow 0^+} \gs \zeta(1+\gs) = L(1,\chi_{O}),$$
we find that
\begin{equation*} \label{sum-over-K-d1}
\lim_{\gs \rightarrow 0^+} \gs \sum_{m \geq 1} \frac{\phi(m)}{[K(E[m]) : K]}m^{-\gs} \leq
\prod_{\ell \mid m_E} \left( 1 - \frac{1}{\ell} \right)^{-1}
\left(
\sum_{m \mid m_E} \frac{\phi(m)}{[K(E[m]) : K]}
\right)
L(1,\chi_{O}). 
\end{equation*}

Note that this inequality holds for any elliptic curve $E/\Q$ with complex multiplication. In particular, 
if we let $j_1, \ldots, j_{13}$ be  the $j$-invariants of  elliptic curves over $\Q$
 with complex multiplication, as per Corollary \ref{13j}, 
 and if, as earlier, for each $1 \leq i \leq 13$ we choose an elliptic curve 
 $E_i/\Q$ with 
 $j(E_i) = j_i$
 and
 $\End_{\overline{\Q}}(E_i) \simeq O_i \not\simeq \Z$,
 denoting by $K_i$ the complex multiplication field of $E_i$
 and applying (\ref{sum-over-K-d1}),
 we obtain that
\begin{equation*}\label{sum-over-K-d1-Ei}
\lim_{\gs \rightarrow 0^+} \gs \sum_{m \geq 1} \frac{\phi(m)}{[K_i(E_i[m]) : K_i]}m^{-\gs} \leq
\prod_{\ell \mid m_{E_i}} \left( 1 - \frac{1}{\ell} \right)^{-1}
\left(
\sum_{m \mid m_{E_i}} \frac{\phi(m)}{[K_i(E_i[m]) : K_i]}
\right)
L(1,\chi_{O_i}).
\end{equation*}

Recalling (\ref{twist-degree}) and arguing as in the first part of the proof, we obtain that 
\begin{equation*}
\lim_{\gs \rightarrow 0^+} \gs \sum_{m \geq 1} \frac{\phi(m)}{[K(E[m]) : K]}m^{-\gs} \leq
6 \max_{1 \leq i \leq 13}
\prod_{\ell \mid m_{E_i}} \left( 1 - \frac{1}{\ell} \right)^{-1}
\left(
\sum_{m \mid m_{E_i}} \frac{\phi(m)}{[K_i(E_i[m]) : K_i]}
\right)
L(1,\chi_{O_i}),
\end{equation*}
which, as before, gives
\begin{equation*} 
C_{d_1}(E)
\leq
6 \max_{1 \leq i \leq 13}
\prod_{\ell \mid m_{E_i}} \left( 1 - \frac{1}{\ell} \right)^{-1}
\left(
\sum_{m \mid m_{E_i}} \frac{\phi(m)}{[K_i(E_i[m]) : K_i]}
\right)
L(1,\chi_{O_i}). 
\end{equation*}
This completes the proof of part (i) of Theorem \ref{constant-nonSerre}. 
 \hfill $\square$

\bigskip

\noindent
{\bf{Proof of part (ii) of Theorem \ref{constant-nonSerre}}}.

\noindent
By Corollary \ref{mE-nonCM}, for any integer $m \geq 1$, written uniquely as $m = m_1 m_2$ with 
$m_1 \mid m_E^{\infty}$ and $(m_2, m_E) = 1$,
we have
\begin{equation}\label{nonCM-vertical}
[\Q(E[m]) : \Q] = [\Q(E[m_1]) : \Q] \cdot |\GL_2(\Z/m_2 \Z)|.
\end{equation}
Furthermore,  for any $\ell$, $d$, with $\ell \mid m_E$, $\ell \nmid d$, and for any $\delta \in \N$, 
by part (ii) of Lemma \ref{vertical-division}
we also have
\begin{eqnarray}\label{kappa-nonCM}
\left[
\Q\left(
E\left[
\ell^{v_\ell(m_E) + \delta} d
\right]
\right) 
:
\Q\left(
E\left[
\ell^{v_\ell(m_E)} d
\right]
\right) 
\right]
= \ell^{4 \delta}.
 \end{eqnarray}

Defining $f(m) := \frac{1}{[\Q(E[m]) : \Q]}$ and $g(m) := \frac{1}{|\GL_2(\Z/m \Z)|}$, 
we apply Lemma \ref{lemma-mult-2} with $M= m_E$ and
$\kappa = 4$.
Note that condition (i) follows immediately from the Chinese Remainder Theorem, 
condition (ii) follows  immediately from  
$$
\ds\sum_{m \geq 1} \frac{1}{|\GL_2(\Z/m \Z)|}
<
\ds\sum_{m \geq 1}
\frac{1}{m^2 \phi(m)^2}
< \infty,
$$
and conditions (iii) and (iv) follow from (\ref{nonCM-vertical}) - (\ref{kappa-nonCM}). We obtain
\begin{eqnarray}\label{sum-over-Q}
C_{d_2}(E)
& \leq & 
 C_{\tau(d_1)}(E)
 \nonumber
 \\
&=&
\ds\sum_{m \geq 1} \frac{1}{[\Q(E[m]) : \Q]}
\nonumber
\\
&\leq &
\zeta(4)
\left(
\ds\sum_{m | m_E} \frac{1}{[\Q(E[m]) : \Q]}
\right)
\left(
\ds\sum_{m \geq 1}
\frac{1}{ |\GL_2(\Z/m \Z)| }
\right)
\nonumber
\\
&\ll&
\ds\sum_{m | m_E} \frac{1}{[\Q(E[m]) : \Q]}.
\end{eqnarray}

To analyze the last sum, recall that
for any integer $m \geq 1$,
denoting by $\zeta_m$ a primitive $m$-th root of unity, we have
$\Q(\zeta_m) \subseteq \Q(E[m])$;
in particular,  we have $\phi(m) \mid [\Q(E[m]) : \Q]$.
Then
\begin{eqnarray*}
\ds\sum_{m \mid m_E} \frac{1}{[\Q(E[m]) : \Q]}
&\leq&
\ds\sum_{m \mid m_E} \frac{1}{\phi(m)} 
\\
&\leq &
\ds\prod_{\ell \mid m_E} \left(1 + \frac{1}{\ell - 1} \ds\sum_{r \geq 0} \frac{1}{\ell^{r}}\right)
\\
&=&
\ds\prod_{\ell\mid m_E} \left(1 + \frac{\ell}{(\ell - 1)^2}\right) 
\\
&=&
\ds\prod_{\ell \mid m_E} \left(1 + \frac{1}{\ell}\right)
\cdot
\ds\prod_{\ell \mid m_E} \left(1 + \frac{2 \ell - 1}{\ell^3 - \ell^2 - \ell + 1}\right)
\\
&\ll&
\ds\prod_{\ell \mid m_E} \left(1 + \frac{1}{\ell}\right)
\\
&\leq&
\exp \left(\ds\sum_{\ell \mid m_E} \frac{1}{\ell}\right),
\end{eqnarray*}
where, in the last line,
we used the elementary inequality $1 + t \leq \exp t$. 
To understand the last sum, we proceed in a standard way. We let $\ell_i$ denote the $i$-th prime
and we recall  that it satisfies the bound $\ell_i \leq i (\log i + \log \log i)$ (see \cite[Theorem 3, p. 69]{RoSc}).
Then
\begin{eqnarray*}
\ds\sum_{\ell \mid m_E} \frac{1}{\ell}
&\leq&
\ds\sum_{i \leq \omega(m_E)} \frac{1}{\ell_i}
\\
&=&
\log \log \ell_{\omega(m_E)} + \O(1)
\\
&\leq&
\log \log \omega(m_E) + \log \log \log \omega(m_E) + \O(1)
\\
&\ll&
\log \log \log m_E,
\end{eqnarray*}
where, in the second line, we  used Mertens' Theorem $\ds\sum_{\ell \leq n} \frac{1}{\ell} = \log \log n + \O(1)$, 
and, in the last line, we used
the bound $\omega(n) \leq \log n$.

The desired bounds for
$C_{d_2}(E)$, $C_{\tau(d_1)}(E)$ are obtained
by plugging the above estimate in (\ref{sum-over-Q})
and
by invoking Proposition \ref{mEboundlemma}.

\bigskip
To estimate $C_{d_1}(E) := C_{d_1, \text{non-CM}}(E)$, 
defining $f(m) := \frac{\phi(m)}{[\Q(E[m]) : \Q]}$ and $g(m) := \frac{\phi(m)}{|\GL_2(\Z/m \Z)|}$, 
we apply Lemma \ref{lemma-mult-2} with $M= m_E$ and
$\kappa = 3$.
Note that condition (i) follows immediately from the Chinese Remainder Theorem, 
condition (ii) follows  immediately from  
$$
\ds\sum_{m \geq 1} \frac{\phi(m)}{|\GL_2(\Z/m \Z)|}
<
\ds\sum_{m \geq 1}
\frac{1}{m^2 \phi(m)}
< \infty,
$$
and conditions (iii) and (iv) follow from (\ref{nonCM-vertical}) - (\ref{kappa-nonCM}). We obtain
\begin{equation}\label{sum-over-Q-d1}
C_{d_1}(E)
=\ds\sum_{m \geq 1} \frac{\phi(m)}{[\Q(E[m]) : \Q]}
\leq
\zeta(3)
\left(
\ds\sum_{m | m_E} \frac{\phi(m)}{[\Q(E[m]) : \Q]}
\right)
\left(
\ds\sum_{m \geq 1}
\frac{\phi(m)}{ |\GL_2(\Z/m \Z)| }
\right).
\end{equation}
Recalling that $\phi(m) \mid [\Q(E[m]) : \Q]$,
we deduce that
\begin{eqnarray*}
\ds\sum_{m \mid m_E} \frac{\phi(m)}{[\Q(E[m]) : \Q]}
&\leq&
\ds\sum_{m \mid m_E} \frac{\phi(m)}{\phi(m)} 
=\ds\sum_{m \mid m_E} 1 
= 
\tau(m_E)
 \ll_\varepsilon m_E^\varepsilon
\end{eqnarray*}
for any $\varepsilon > 0$.
The desired bound for
$C_{d_1}(E)$
is obtained
by plugging the above estimates in (\ref{sum-over-Q-d1})
and
by invoking Proposition \ref{mEboundlemma}.
\hfill $\square$

\bigskip
\section{Constants for Serre curves: proof of Theorem \ref{serrecurveconstant}}

In this section we prove closed formulae relating the constants of Conjectures \ref{d-one} - \ref{d-two}
to the idealized constants (\ref{c-d-one}) - (\ref{c-d-two}), as stated in Theorem \ref{serrecurveconstant}.
The key ingredients in the proof
are Corollary \ref{mE-nonCM} and the following lemma:
\begin{lemma}\label{lemma-mult-1}
Let 
$f, g : \N \backslash \{0\} \longrightarrow \mbc^\times$ 
be arithmetic functions satisfying the following:
\begin{enumerate} 
\item[(i)] 
$g$ is multiplicative; 
\item[(ii)]
$\ds\sum_{m \geq 1} |g(m)|$ converges.
\end{enumerate}
Assume that 
$\exists \; M \in \N \backslash \{0\}$,
$\exists \; \alpha \in (0, \infty)$
and
$\exists \; \kappa \in \N \backslash \{0, 1\}$
such that:
\begin{enumerate}
\item[(iii)]
$\forall m \in \N \backslash \{0\}$, we have
\begin{equation*}
f(m) = 
\begin{cases}
\alpha  g(m) & \text{if} \ M \mid m, 
\\ 
g(m) &  \ \text{else}; 
\end{cases} 
\end{equation*}
\item[(iv)]
$\forall \; m \mid M^{\infty}$, we have
$
g(m M) = m^{-\kappa} g(M).
$
\end{enumerate}
Then 
$$
\ds\sum_{m \geq 1} f(m) 
= 
\left( 1 + (\alpha-1)g(M) \ds\prod_{\ell \mid M} g_\ell^{-1}(1-\ell^{-\kappa})^{-1} \right) \ds\prod_\ell g_\ell,
 $$
 where, 
 for any rational prime $\ell$, 
$g_\ell := \ds\sum_{r\geq 0} g(\ell^r)$.
\end{lemma}

\begin{proof}
This result can be  found in \cite[Lemma 3.12]{Ko}; we give its proof for completeness:
\begin{eqnarray*}
\ds\sum_{m \geq 1} f(m) 
&= &
\ds\sum_{m \geq 1} g(m) + (\alpha-1) \ds\sum_{M \mid m} g(m) 
\hspace*{5cm}
\text{(by (iii))}
\\
&=& 
\ds\prod_\ell g_\ell + (\alpha-1) \sum_{M \mid m} g(m)
\hspace*{5.7cm} 
\text{(by (i))}
\\
&= &
\ds\prod_\ell g_\ell + (\alpha-1) \ds\sum_{m \geq 1} g(m M) 
\\
&= &
\ds\prod_\ell g_\ell + (\alpha-1) 
\ds\sum_{
m_1 \mid M^\infty 
\atop 
(m_2, M) = 1
} 
g(m_1 m_2 M)
 \\
&= &
\ds\prod_\ell g_\ell 
+ 
(\alpha-1) 
\left( \ds\sum_{m_1 \mid M^\infty} g(m_1 M) \right)
 \left(\ds\sum_{(m_2, M) = 1} g(m_2) \right)
\hspace*{1.2cm}  
\text{(by (i))}
 \\
&=  &
\ds\prod_\ell g_\ell 
+ 
(\alpha-1) 
 \left(\ds\prod_{\ell \nmid M} g_\ell \right)
\left( g(M) \ds\sum_{m_1 \mid M^\infty} m_1^{-\kappa}  \right)
 \hspace*{1cm} 
 \text{(by (i)) and (iv))}
\\
&=  &
\ds\prod_\ell g_\ell 
+ 
(\alpha-1) 
\left( g(M) \ds\prod_{\ell \mid M} 
\left(\ds\sum_{r\geq 0} \ell^{-r\kappa } \right)  \right) 
\left(\ds\prod_{\ell \nmid M}  g_\ell \right)
\\
&=  &
\ds\prod_\ell g_\ell 
+ 
(\alpha-1) 
\left( g(M) \ds\prod_{\ell \mid M} (1-\ell^{-\kappa})^{-1}  \right) 
\left(\ds\prod_{\ell \nmid M} g_\ell \right)
\\
&=  &
\left( 1 + (\alpha-1) g(M) \ds\prod_{\ell \mid M} g_\ell^{-1}(1-\ell^{-\kappa})^{-1}  \right) \ds\prod_\ell g_\ell. 
\end{eqnarray*}

\end{proof}

\medskip

\noindent
{\bf{Proof of Theorem \ref{serrecurveconstant}}}.

\noindent
The proposition follows from 
part (iv) of Proposition \ref{mESerre},
Lemma  \ref{lemma-mult-1}
with $M = m_E$
and $f$, $g$, $\alpha$, $\kappa$ as described below,
and 
from summing geometric series.
Note that, by definition, $\ds\prod_\ell g_\ell$ is the idealized constant $C_{d_1}$, $C_{\tau(d_1)}$,  respectively $C_{d_2}$.

\medskip

\begin{center}
\begin{tabular}{c | c | c | c | c | c}
& $f(m)$ & $g(m)$ & $\alpha$ & $\kappa$ & $g_\ell$ 
\\ 
\hline
 $C_{d_1, non-CM}(E)$ & $\frac{\phi(m)}{[\Q(E[m]):\Q]}$ & $\frac{\phi(m)}{|\GL_2(\Z/m\Z)|}$& $\alpha = 2$ &$\kappa = 3$ & $1 + \frac{1}{\ell^3(1-\ell^{-2})(1-\ell^{-3})}$ 
 \\ 
 & & & &
 \\
  $C_{\tau(d_1)} (E)$ & $\frac{1}{[\Q(E[m]):\Q]}$ & $\frac{1}{|\GL_2(\Z/m\Z)|}$& $\alpha = 2$ &$\kappa = 4$ & $1 + \frac{1}{\ell^4(1-\ell^{-1})(1-\ell^{-2})(1-\ell^{-4})}$ 
 \\
 &  & & &
 \\
 $C_{d_2}(E)$ & $\frac{(-1)^{\omega(m)}\phi(\rad(m))}{m[\Q(E[m]):\Q]}$ & $\frac{(-1)^{\omega(m)}\phi(\rad(m))}{m|\GL_2(\Z/m\Z)|}$& $\alpha = 2$ &$\kappa = 5$ & $1 - \frac{1}{\ell^4(1-\ell^{-2})(1-\ell^{-5})}$ 
 \\
\end{tabular}
\end{center}

\hfill $\square$

\bigskip
\section{Averaging the constants over a family: proof of Theorem \ref{main-thm}}

In this section we prove Theorem \ref{main-thm} using the results of Sections 2-4 and following the approach initiated in \cite{Jo09}.
For this, let $A, B > 2$ and $n \in \N \backslash \{0\}$ be fixed and consider the  moment
\begin{eqnarray*}
\frac{1}{|{\mathcal{C}}(A, B)|}
\ds\sum_{E \in {\mathcal{C}}(A, B)  }
|C(E) - C|^n,
\end{eqnarray*}
where the pair $(C(E), C)$ is, respectively,
$(C_{d_1}(E), C_{d_1})$, 
$(C_{\tau(d_1)}(E), C_{\tau(d_1}))$,
and
$(C_{d_2}(E), C_{d_2})$, 
and where
$C_{d_1}(E)$ is 
$C_{d_1, \text{non-CM}}(E)$ if $E$ is without complex multiplication, 
and 
$C_{d_1, \text{CM}}(E)$ if $E$ is with complex multiplication.
Our strategy is to partition ${\mathcal{C}}(A, B)$ into 
the subset of elliptic curves with complex multiplication (CM curves),  
the subset of elliptic curves without complex multiplication and which are not Serre curves
(non-CM $\&$ non-Serre curves), 
and  the subset of Serre curves;
we then estimate each emerging subsum via different techniques.\\

 
To handle the contribution from elliptic curves with complex multiplication, using
 part (i) of Theorem \ref{constant-nonSerre} and  Lemma \ref{lemma-numcmcurves},
we obtain
\begin{eqnarray}\label{averageCM}
\frac{1}{|{\mathcal{C}}(A, B)|}
\ds\sum_{
E \in {\mathcal{C}}(A, B)  
\atop{E \ \text{CM}}
}
|C(E) - C|^n
\ll_n
\frac{1}{|{\mathcal{C}}(A, B)|}
\ds\sum_{
E \in {\mathcal{C}}(A, B)  
\atop{E \ \text{CM}}
}
1
\ll
\frac{1}{A} + \frac{1}{B}.
\end{eqnarray}


To handle the contribution from 
elliptic curves without complex multiplication, which are not Serre curves,
using part (ii) of Theorem \ref{constant-nonSerre}, 
as well as Theorem  \ref{serre-curves},
we obtain
\begin{eqnarray}\label{averagenonCMd2}
&&
\frac{1}{|{\mathcal{C}}(A, B)|}
\ds\sum_{
E \in {\mathcal{C}}(A, B)  
\atop{E \ \text{non-CM},
\atop{
\; \; \; \text{non-Serre}
}
}
}
|C(E) - C|^n
\nonumber
\\
&\ll_n&
\frac{
\left(
\log \log 
 \left\{
 \max\{A^3, B^2\} \cdot \log(\max\{A^3, B^2\})^{\gamma}
 \right\}
\right)^n
\cdot
(\log \min\{A, B\})^{\gamma'}
}{
\sqrt{\min\{A, B\}}
}
\end{eqnarray}
for $(C(E), C)$ equal to 
$(C_{\tau(d_1)}(E), C_{\tau(d_1}))$ or
$(C_{d_2}(E), C_{d_2})$, 
and
\begin{eqnarray}\label{averagenonCMd1}
&&
\frac{1}{|{\mathcal{C}}(A, B)|}
\ds\sum_{
E \in {\mathcal{C}}(A, B)  
\atop{E \ \text{non-CM},
\atop{
\; \; \; \text{non-Serre}
}
}
}
|C(E) - C|^n
\nonumber
\\
&\ll_{n, \varepsilon}&
\frac{
\left(
 \max\{A^3, B^2\} \cdot \log(\max\{A^3, B^2\})^{\gamma}
\right)^{n \varepsilon}
\cdot
(\log \min\{A, B\})^{\gamma'}
}{
\sqrt{\min\{A, B\}}
}
\end{eqnarray}
for $(C(E), C)$ equal to 
$(C_{d_1, \text{non-CM}}(E), C_{d_1})$. \\

To handle the contribution from 
 Serre curves,
using Theorem \ref{serrecurveconstant} and part (iii) of Proposition \ref{mESerreProp},
we obtain
\begin{eqnarray*}
&&
\frac{1}{|{\mathcal{C}}(A, B)|}
\ds\sum_{
E \in {\mathcal{C}}(A, B)  
\atop{E \ \text{Serre}}
}
|C(E) - C|^n
\ll
\frac{C^n}{|{\mathcal{C}}(A, B)|}
\ds\sum_{
E \in {\mathcal{C}}(A, B)  
\atop{E \ \text{Serre}}
}
\frac{1}{|\Delta_{\text{sf}}(E)|^{3 n}}.
\end{eqnarray*}
Next, following the approach of  \cite[Section 4.2]{Jo09},
we choose a parameter $z = z(A, B)$ and partition  the Serre curves $E$ in ${\mathcal{C}}(A, B)$ according to whether 
$ |\Delta_{\text{sf}}(E)| \leq z$ or not.
For the first subsum we use \cite[Lemma 22, p. 705]{Jo09},
while for the second subsum we use trivial estimates:
\begin{eqnarray*}
&&
\frac{C^n}{|{\mathcal{C}}(A, B)|}
\ds\sum_{
E \in {\mathcal{C}}(A, B)  
\atop{E \ \text{Serre}
\atop{
|\Delta_{\text{sf}}(E)| \leq z
}
}
}
\frac{1}{|\Delta_{\text{sf}}(E)|^{3 n}}
+
\frac{C^n}{|{\mathcal{C}}(A, B)|}
\ds\sum_{
E \in {\mathcal{C}}(A, B)  
\atop{E \ \text{Serre}
\atop{
|\Delta_{\text{sf}}(E)| > z
}
}
}
\frac{1}{|\Delta_{\text{sf}}(E)|^{3 n}}
\\
&\ll_n&
\frac{1}{A B}
\left(
\#\left\{
(a, b) \in \Z^2:
|a| \leq A, |b| \leq B,
4a^3 + 27 b^2 \neq 0,
|(4a^3 + 27 b^2)_{\text{sf}}| \leq z
\right\}
+
\frac{1}{z^{3 n} }
\right)
\\
&\ll&
\frac{1}{A}
+
 \frac{z (\log A)^7 (\log B)}{B}
+
\frac{1}{z^{3 n} A B}.
\end{eqnarray*}
Upon choosing
$$
z \asymp \left( \frac{1}{A (\log A)^7 (\log B)} \right)^{\frac{1}{3 n + 1}},
$$
we deduce that
\begin{eqnarray}\label{averageSerre}
&&
\frac{1}{|{\mathcal{C}}(A, B)|}
\ds\sum_{
E \in {\mathcal{C}}(A, B)  
\atop{E \ \text{Serre} }
}
|C(E) - C|^n
\ll_n
\frac{1}{A}
+
\frac{
(\log A)^{\frac{21 n}{3 n + 1}}
(\log B)^{\frac{3 n}{3 n + 1}}
}{
A^{\frac{1}{3 n + 1}}
B
}.
\end{eqnarray}
\\
Putting together (\ref{averageCM}), (\ref{averagenonCMd2}),  (\ref{averagenonCMd1}), and (\ref{averageSerre}), it suffices to show that if $A = A(x)$ and $B = B(x)$ are functions tending to infinity with $x$ and such that 
\begin{equation}\label{log_ratio_bounded} 
\limsup_{x \to \infty} \  \frac{\log{A(x)}}{\log{B(x)}} < \infty, 
\quad
\limsup_{x \to \infty} \ \frac{\log{B(x)}}{\log{A(x)}} < \infty, \end{equation}
then there exists an $\varepsilon > 0$ such that 
\[ \ds\lim_{x \rightarrow \infty}
\frac{
\left(
 \max\{A(x)^3, B(x)^2\} \cdot \log(\max\{A(x)^3, B(x)^2\})^{\gamma}
\right)^{\varepsilon}
\cdot
(\log \min\{A(x), B(x)\})^{\gamma'}
}{
\sqrt{\min\{A(x), B(x)\}}
}
= 0. \]
Assumption \eqref{log_ratio_bounded} implies that 
\[ \limsup_{x\to \infty} \ \frac{\log \left(\max\{A(x)^3,  B(x)^2 \}\right)}{\log \left(\min\{A(x), B(x)\}\right)} < \infty, \]
and so
\begin{equation*}
\limsup_{x\to \infty} \ \frac{\log \left(\max\{A(x)^3,  B(x)^2 \}\cdot \log(\max\{A(x)^3, B(x)^2\})^{\gamma} \right)}{\log \left(\sqrt{\min\{A(x), B(x)\}}\cdot (\log \min\{A(x), B(x)\})^{-\gamma'}\right)} \equalscolon \varepsilon_0 < \infty. 
\end{equation*}
For any $\varepsilon < \varepsilon_0$, we therefore have
\begin{equation*}
\hspace*{-0.5cm}
\lim_{x \to \infty} \varepsilon \cdot
 \log \left(\max\{A(x)^3,  B(x)^2 \}\cdot \log(\max\{A(x)^3, B(x)^2\})^{\gamma} \right) - \log \left(\sqrt{\min\{A(x), B(x)\}}\cdot (\log \min\{A(x), B(x)\})^{-\gamma'}\right) 
 = -\infty. 
\end{equation*}
Exponentiating finishes the argument.


%
%
%
%
%
%
%
%
%
%
%
%
%
%

\bigskip
\noindent
{\bf{Acknowledgments.}} 
This research started during the {\emph{Arizona Winter School 2016: Analytic Methods in Arithmetic Geometry}}, organized at the 
University of Arizona, Tucson, USA, during March 12-16, 2016.
We thank the conference organizers Alina Bucur, David Zureick-Brown, Bryden Cais, Mirela Ciperiani, and Romyar Sharifi for
all their time and support, and we thank the National Science Foundation for sponsoring our participation in this conference.



\bigskip
\bigskip
\bigskip


{\small{

}}

\end{document}